\documentclass[11pt,leqno]{amsart}

\usepackage{amsmath,amsthm}
\usepackage{latexsym}
\usepackage{amssymb}
\usepackage{mathtools}
\usepackage[parfill]{parskip}
\usepackage{mathrsfs}
\usepackage{enumerate}
\usepackage{booktabs} % For tables
\usepackage{hyperref} % For links in bib

\newtheorem{theorem}{Theorem}

\newtheorem{corollary}[theorem]{Corollary}

\newtheorem{lemma}[theorem]{Lemma}
\newtheorem*{claim*}{Claim}

\theoremstyle{definition}
\newtheorem{definition}[theorem]{Definition}

\newtheorem{example}[theorem]{Example}

\numberwithin{theorem}{section}

\makeatletter
\@namedef{subjclassname@2020}{\textup{2020 {\it Mathematics Subject Classification}}}
\makeatother

\begin{document}

\bibliographystyle{plain}

\title {Designs related through projective and Hopf maps}
\author {Ayodeji Lindblad}
\address{Department of Mathematics, Massachusetts Institute of Technology, Cambridge, Massachusetts 02139}
\email{ayodeji@mit.edu}
%\subjclass[2020]{05B30 (primary); 41A55; 41A63; 52C35; 65D32 (secondary)}
%\keywords{Spherical $t$-design, projective $t$-design, cubature, quadrature, projective map, Hopf map}
\maketitle

\begin{abstract}
We verify a construction which, for $\Bbb K$ the reals, complex numbers, quaternions, or octonions, builds a spherical $t$-design by placing a spherical $t$-design on each $\Bbb K$-projective or $\Bbb K$-Hopf fiber associated to the points of a $\lfloor t/2\rfloor$-design on a quotient projective space $\Bbb{KP}^n\neq\Bbb{OP}^2$ or sphere. This generalizes work of K\"{o}nig and Kuperberg, who verified the $\Bbb K=\Bbb C$ case of the projective settings, and of Okuda, who (inspired by independent observation of this construction by Cohn, Conway, Elkies, and Kumar) verified the $\Bbb K=\Bbb C$ case of the generalized Hopf settings.
\end{abstract}

\section{Background and main results}
\label{sec:intro}

% Background on spherical designs
%TODO check
Spherical $t$-designs are finite subsets of spheres introduced by Delsarte, Goethals, and Seidel \cite[Section 5]{Delsarte...77} which provide good global approximations of the spheres they lie on, in the sense that polynomials of degree at most $t$ have the same average value on a $t$-design as on the sphere. $t$-designs on the $d$-dimensional unit sphere $S^d$ in $\Bbb R^{d+1}$ were shown to exist for all $t$ and $d$ by Seymour and Zaslavsky \cite[Corollary 1]{SeymourZalavsky84} and later explicitly constructed by Bajnok as the direct product of $d-2$ weighted interval $t$-designs (i.e. finite subsets of the weighted interval which exactly average polynomials of degree at most $t$) and a regular $(t+1)$-gon \cite{Bajnok92}. Regular $(t+1)$-gons are known to be the only optimally small $t$-designs on $S^1$ for all $t$ \cite[Example 5.1.4]{Delsarte...77}, but in any other dimension, only a handful of optimally small spherical $t$-designs are known. %TODO give citation discussing this - maybe Bannai-Bannai?
It was only in the past decade that a landmark paper of Bondarenko, Radchenko, and Viazovska resolved the longstanding open problem of finding the asymptotic order as $t\to\infty$ of optimally small $t$-designs on $S^d$ for $d>1$. Specifically, these authors proved that there exist $t$-designs on $S^d$ comprised of $N$ points for any $t$, $d$, and $N\geq C_dt^d$ ($C_d$ a constant depending on $d$) \cite[Theorem 1]{Bondarenko...13}, showing that lower bounds \cite[Theorems 5.11, 5.12]{Delsarte...77} on the sizes of spherical $t$-designs provided by Delsarte, Goethals, and Seidel are asymptotically optimal as $t\to\infty$. Considerable exploration of the properties of spherical designs---much of which is presented in a survey article by Bannai and Bannai \cite{BannaiBannai09}---has been done past these results since their introduction.

% Background on projective designs
%TODO check
%TODO discuss first constructions of projective designs and the current status of known results? Note we only know a few optimal designs in this setting too
Projective $t$-designs are finite subsets of projective spaces which, in analogy with spherical $t$-designs, exactly average degree at most $t$ polynomials on the projective spaces they lie on. These objects were first considered when Neumaier \cite{Neumaier81} introduced designs on \emph{Delsarte spaces}, a class of metric spaces including projective spaces. As in the case of spherical designs, it follows from the work of Seymour and Zaslavsky \cite[Main Theorem]{SeymourZalavsky84} that $t$-designs exist on any projective space for any $t$. Following the results of Bondarenko, Radchenko, and Viazovska in the spherical setting \cite[Theorem 1]{Bondarenko...13}, Etayo, Marzo, and Ortega-Cerd\`a proved that there exist $t$-designs on any $d$-dimensional algebraic manifold $M$ (so notably, on any projective space of real dimension $d$) comprised of $N$ points for any $t$, $d$, and $N\geq C_Mt^d$ ($C_M$ a constant depending on $M$) \cite[Theorem 2.2]{Etayo...18}, showing that lower bounds of Dunkl \cite{Dunkl79} and of Bannai and Hoggar \cite[Theorem 1]{BannaiHoggar85} %TODO verify this is a reasonable way to discuss the Dunkl and Bannai-Hoggar bounds
on the sizes of $t$-designs on complex and quaternionic projective spaces analogous to the bounds of Delsarte, Goethals, and Seidel \cite[Theorems 5.11, 5.12]{Delsarte...77} in the spherical setting are asymptotically optimal as $t\to\infty$.

% Constructions of interest
%TODO check
Spherical and complex projective designs can be related by the construction introduced by K\"{o}nig \cite[Corollary 1]{Koning98} and further investigated by Kuperberg \cite[Theorem 4.1]{Kuperberg06} which builds a $t$-design on $S^{2n+1}$ by placing the vertices of a regular $(t+1)$-gon on each fiber of the complex projective map associated to a $\lfloor t/2\rfloor$-design on $\Bbb{CP}^n$. A related construction which builds a $t$-design on $S^3$ by placing the vertices of a regular $(t+1)$-gon on each fiber of the Hopf map associated to a $\lfloor t/2\rfloor$-design on $S^2$ was---inspired by the independent observation of the construction of K\"{o}nig by Cohn, Conway, Elkies, and Kumar \cite[Section 4]{Cohn...06}---verified by Okuda \cite[Theorem 1.1]{Okuda15}. We generalize these results to prove Theorem \ref{thm:bigthm}, whose main consequence is that one can build a spherical $t$-design by placing a spherical $t$-design on each fiber of the projective or generalized Hopf map associated to a $\lfloor t/2\rfloor$-design on a quotient real, complex, quaternionic, or octonionic projective space $\Bbb{KP}^n\neq\Bbb{OP}^2$ or sphere. This gives new constructions of $t$-designs on $S^{4n+3}$ from $\lfloor t/2\rfloor$-designs on $\Bbb{HP}^n$ and $t$-designs on $S^3$ for any $n\in\Bbb N_+$, $t$-designs on $S^7$ from $\lfloor t/2\rfloor$-designs on $S^4$ and $t$-designs on $S^3$, and $t$-designs on $S^{15}$ from $\lfloor t/2\rfloor$-designs on $\Bbb{OP}^1$ or $S^8$ and $t$-designs on $S^7$.

% Main theorem
%TODO check
\begin{theorem}[Main Theorem]\label{thm:bigthm}
For $\Bbb K$ the reals $\Bbb R$, complex numbers $\Bbb C$, quaternions $\Bbb H$, or octonions $\Bbb O$ and $k:=\dim_\Bbb R\Bbb K-1$, take $\Pi:S^d\to\Sigma$ to be either the $\Bbb K$-projective map 
\begin{equation}\label{eq:proj}
\Pi_\Bbb K:S^{(k+1)n+k}\subset\Bbb K^{n+1}\to\Bbb{KP}^n,\quad\omega\mapsto[\omega]
\end{equation}
for $n\in\Bbb N$ ($n=1$ if $\Bbb K=\Bbb O$) or the $\Bbb K$-Hopf map
\begin{equation}\label{eq:hopf}
\pi_\Bbb K:S^{2k+1}\subset\Bbb K^2\to S^{k+1}\subset\Bbb R\times\Bbb K,\quad(a,b)\mapsto (|a|^2-|b|^2,2a\overline b).
\end{equation}
For weighted 0-designs $(Y,\lambda_Y)$ on $\Sigma$ and $(Z_y,\lambda_y)$ on $S^k$ alongside base points $z_y\in\Pi^{-1}(y)$ for each $y\in Y$, we consider
\begin{equation}\label{eq:X}
X:=\bigcup_{y\in Y}\{z_yz\:|\:z\in Z_y\},\quad\lambda(z_yz):=\lambda_Y(y)\lambda_y(z)\quad\text{\it for}\quad z\in Z_y.
\end{equation}
For any $t\in\Bbb N$, $(Y,\lambda_Y)$ is a weighted $\lfloor t/2\rfloor$-design on $\Sigma$ if $(X,\lambda)$ is a weighted $t$-design on $S^d$. Moreover, $(X,\lambda)$ is a weighted $t$-design on $S^d$ if $(Y,\lambda_Y)$ is a weighted $\lfloor t/2\rfloor$-design on $\Sigma$ and $(Z_y,\lambda_y)$ is a weighted $t$-design on $S^k$ for each $y\in Y$. If $|Z_y|$ is constant in $y\in Y$, the corresponding statements for unweighted designs (weighted designs with constant weight function) hold. 
\end{theorem}

% Consequence of two applications of the main theorem
An interesting consequence of two applications of Theorem \ref{thm:bigthm} is that, for $\Bbb L\subset\Bbb K$ the reals, complex numbers, or quaternions and $l:=\dim_\Bbb R\Bbb L-1$, we can build a $t$-design on $\Bbb {LP}^{(d-l)/(l+1)}$ by placing a $t$-design on $\Pi_{\Bbb L}(w)\cong\Bbb{LP}^{(k-l)/(l+1)}$ for each point $w\subset S^{d}$ of a $t$-design on $\Bbb{KP}^n$ as in \eqref{eq:X}.

% Layout of the paper
%TODO check
We define the objects of interest in Subsection \ref{sub:prelims}. We then prove the projective settings of Theorem \ref{thm:bigthm} in Section \ref{sec:proofs} and use these results to verify the generalized Hopf settings of the theorem in Subsection \ref{sub:hopf}. Finally, we provide examples---some of which are known to be optimally small---of unweighted spherical $t$-designs built as in Theorem \ref{thm:bigthm} and discuss the sense in which the constructions of the theorem are each asymptotically optimal as $t\to\infty$ in Section \ref{sec:ex}.

\subsection{Spaces, maps, and designs}
\label{sub:prelims}

% The reals, complex numbers, quaternions, and octonions
%TODO check
Consider the real numbers $\Bbb R$, the complex numbers $\Bbb C$, the quaternions $\Bbb H$, and the octonions $\Bbb O$, which Hurwitz's theorem \cite[Section 10]{Ebbinghaus...91} states are exactly the normed division algebras over the real numbers. These algebras are respectively generated over $\Bbb R$ %TODO good phrasing?
by elements $\{e_0\}$, $\{e_0,e_1\}$, $\{e_0,e_1,e_2,e_3\}$, and $\{e_i\}_{i=0}^7$ satisfying $e_0^2=e_0$, $e_i^2=-e_0$ for $i\neq0$, and the relations illustrated by the Fano plane in Figure \ref{fig:fano}.
While $\Bbb R$ and $\Bbb C$ are commutative and associative, $\Bbb H$ is associative but noncommutative, and $\Bbb O$ is noncommutative and nonassociative but \emph{alternative}---i.e., for all $o_1,o_2\in\Bbb O$, $o_1(o_1o_2)=(o_1o_1)o_2$ and $o_1(o_2o_2)=(o_1o_2)o_2$ \cite{Baez01}. 

% Fano plane
\begin{figure}
\begin{center}
\includegraphics[width=.6\textwidth]
{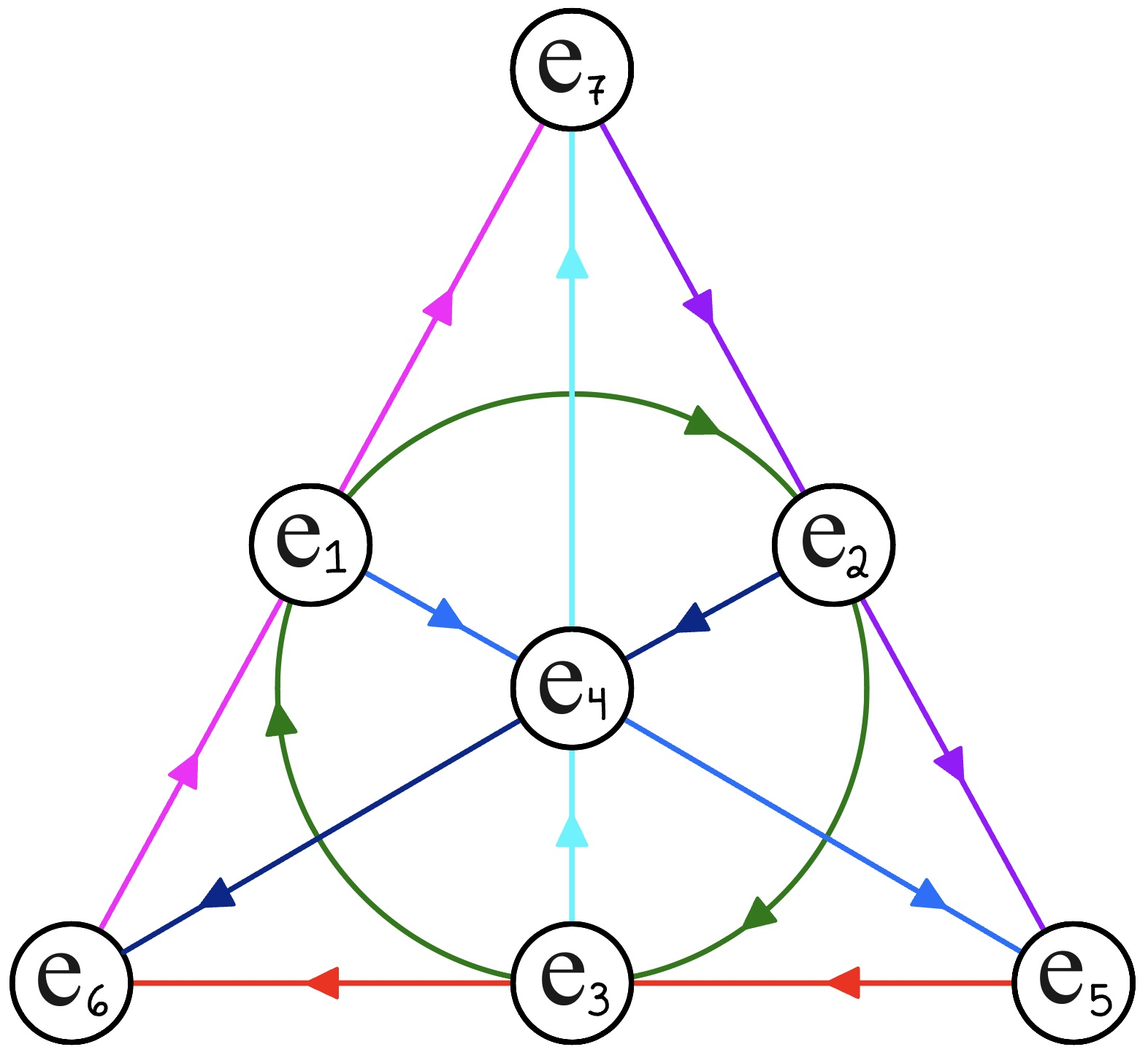}
\caption{\label{fig:fano} The Fano plane, which visualizes how to multiply octonions. We consider 7 lines in this picture: the 3 sides of the triangle, its 3 altitudes, and the circle inscribed in the triangle. If arrows point from $e_i$ to $e_j$ to $e_l$ along any one of these lines, we have $e_i e_j=-e_j e_i=e_l$.}
\end{center}
\end{figure}

% Vector spaces, spheres, and projective spaces
%TODO check
Fix $m\in\Bbb N$. By $\Bbb R^{m+1}$, $\Bbb C^{m+1}$, $\Bbb H^{m+1}$, and $\Bbb O^{m+1}$, we respectively denote the $(m+1)$-dimensional real, complex, quaternionic, and octonionic vector spaces. $S^m$ is taken to be the $m$-dimensional unit sphere in $\Bbb R^{m+1}$ and is equipped with the uniform measure (meaning a measure on a metric space with respect to which any two balls of the same radius have the same measure) $\sigma^m$, which we abbreviate by $\sigma$ and normalize such that $\sigma(S^m)=1$. Then, picking
\begin{equation}\label{eq:Kknd}
\begin{gathered}
\Bbb K\in\{\Bbb R,\Bbb C,\Bbb H,\Bbb O\},\quad k:=\dim_\Bbb R(\Bbb K)-1,\\
n\in\Bbb N\quad(n=1\quad\text{\it{if}}\quad\Bbb K=\Bbb O),\quad d:=(k+1)n+k,
\end{gathered}
\end{equation}
$S^d$ is the unit sphere in $\Bbb{K}^{n+1}$. We define the $(k+1)n$-dimensional (over $\Bbb R$) $\Bbb K$-projective space $\Bbb{KP}^n$---the space of $\Bbb K$-lines in $\Bbb{K}^{n+1}$---using a right action of multiplication, setting
\begin{equation}\label{eq:KPn}
\Bbb{KP}^n:=\{[\omega]\:|\:\omega\in S^d\},\quad[\omega]:=\begin{cases}
\left\{\frac{(\omega\overline\omega_{n+1})\zeta}{|\omega_{n+1}|}\:\middle|\:\zeta\in S^k\right\} & \omega_{n+1}\neq0 \\
\{\omega\zeta\:|\:\zeta\in S^k\} & \omega_{n+1}=0
\end{cases}.
\end{equation}
Note that when $\Bbb K\in\{\Bbb R,\Bbb C,\Bbb H\}$, we may use the simpler presentation
\[[\omega]=\{\omega\zeta\:|\:\zeta\in S^k\}\]
by associativity of these algebras.
We equip $\Bbb{KP}^n$ with the uniform measure $\rho^n_\Bbb K$, which we abbreviate by $\rho$ and normalize such that $\rho(\Bbb{KP}^n)=1$.

% Definition of OP1 is standard
%TODO check
Our definition of $\Bbb{OP}^1$ is equivalent to the more common presentation
\[\left\{\begin{bmatrix}a\\
b\end{bmatrix}\begin{bmatrix}\overline a & \overline b\end{bmatrix}\:\middle|\:(a,b)\in S^{15}\subset\Bbb O^2\right\}\]
of $\Bbb{OP}^1$ discussed by Baez \cite[Subsection 3.1]{Baez01}. Specifically, we can observe that $|a|^2=|\alpha|^2$, $a\overline b=\alpha\overline\beta$, and $|b|^2=|\beta|^2$ exactly when $[(a,b)]=[(\alpha,\beta)]$ for $(a,b),(\alpha,\beta)\in S^{15}$ using the fact that the subalgebra generated by any two elements of $\Bbb O$ is associative, as can be seen from alternativity of $\Bbb O$ combined with a lemma of Artin \cite[Section 3]{Schafer66}.
Baez also discusses why it is sensible to consider octonionic projective spaces $\Bbb{OP}^n$ for exactly $n\in\{1,2\}$. The proofs in Section \ref{sec:proofs} do not directly apply to give a construction which builds a $t$-design on $S^{23}$ from a $\lfloor t/2\rfloor$-design on $\Bbb{OP}^2$ and $t$-designs on $S^7$ because $\Bbb{OP}^2$ does not admit a presentation analogous to \eqref{eq:KPn}. Instead, they directly verify a construction which uses a map $\Pi_\Bbb O$ analogous to \eqref{eq:proj} from the set $S(\Bbb O^3_A)$ of triples $(a,b,c)\in S^{23}\subset\Bbb O^3$ satisfying $(ab)c=a(bc)$ to $\Bbb{OP}^2$ to relate $t$-designs on $S(\Bbb O^3_A)$ to $\lfloor t/2\rfloor$-designs on $\Bbb{OP}^2$ and $t$-designs on preimages of $\Pi_\Bbb O$. %TODO maybe elaborate, but it seems best to me to leave it at this for now

% Projective maps
%TODO check
Consider the $\Bbb K$-projective maps \eqref{eq:proj} and $\Bbb K$-Hopf map \eqref{eq:hopf} as in Theorem \ref{thm:bigthm}. The measure $\rho$ we consider on $\Bbb{KP}^n$ is equivalent to the pushforward $(\Pi_\Bbb K)_*\sigma=\sigma\circ\Pi_\Bbb K^{-1}$ of the measure $\sigma$ we consider on $S^d$ by the $\Bbb K$-projective map, as can be seen since $(\Pi_\Bbb K)_*\sigma$ will be uniform by symmetry and will be normalized such that $((\Pi_\Bbb K)_*\sigma)(\Bbb{KP}^n)=\sigma(S^d)=1$. Additionally, we have $\pi_\Bbb K=h_{\Bbb K}\circ\Pi_\Bbb K$, where
\begin{equation}\label{eq:hK} %TODO make sure this is used, don't repeat later
h_{\Bbb K}:\Bbb{KP}^1\to S^{k+1}\subset\Bbb R\times\Bbb K,\quad[(a,b)]\mapsto(|a|^2-|b|^2,2a\overline b)
\end{equation}
is a diffeomorphism identifying $\Bbb{KP}^1$ with $S^{k+1}$. We note that $\sigma=(h_{\Bbb K})_*\rho$, so we also have $\sigma^{k+1}=(\pi_\Bbb K)_*\sigma^{2k+1}$.
Preimages of the $\Bbb K$-projective and $\Bbb K$-Hopf maps can be seen from \eqref{eq:proj} to each arise as the unit sphere in a subspace of $\Bbb K^{n+1}$ of dimension 1 over $\Bbb K$. We equip these fibers with the uniform measure $\sigma$, which we normalize such that each fiber has volume 1.

% Polynomials
%TODO check
The space of real-valued polynomials on $\Bbb R^{m+1}$ is the closure of the constant and coordinate functions $\Bbb R^{m+1}\to\Bbb R$ under addition and multiplication. We then say that $P_t(\Bbb R^{m+1})$ is the space of polynomials of degree less than or equal to $t$ on $\Bbb R^{m+1}$; i.e. polynomials whose terms are each the product of a constant function and at most $t$ coordinate functions. We then define
\begin{gather}\label{eq:Sdpoly}
P_t(S^m):=\{f|_{S^m}\:|\:f\in P_t(\Bbb R^{m+1})\},\\
\label{eq:KPnpoly}
P_t(\Bbb{KP}^n):=\{f:\Bbb{KP}^n\to\Bbb R\:|\:\Pi_\Bbb K^*f\in P_{2t}(S^d)\},
\end{gather}
where $\Pi_\Bbb K^*f$ is the pullback $f\circ\Pi_\Bbb K$ of $f$ by $\Pi_\Bbb K$. We consider coordinates in $\Bbb K$ for real-valued polynomials via the association $\Bbb K\cong\Bbb R^{k+1}$ which realizes $\Bbb K$ as a $(k+1)$-dimensional real vector space with basis vectors $\{e_i\}_{i=0}^k$, the standard generators for the algebra $\Bbb K$ over $\Bbb R$.

We do not use explicit spanning sets of polynomials on projective spaces in the proofs below, but note that Lyubich and Shatalova \cite[Main Theorem]{LyubichShatalova05} proved that for $\Bbb K\in\{\Bbb C,\Bbb H\}$ and where $\langle a,b\rangle=\sum_{i=1}^{n+1}a_i\overline{b_i}$ is the conventional inner product on $\Bbb K^{n+1}$, $P_t(\Bbb{KP}^n)$ is spanned over $\Bbb K$ by the functions
\[\left\{[\omega_1]\mapsto|\langle \omega_1,\omega_2\rangle|^s\:\middle|\:[\omega_2]\in\Bbb{KP}^n,\:s\in\{0,2,\dots,2t\}\right\}.\]
%This result should also hold when $\Bbb K=\Bbb O$ and $n=1$.
%TODO note that this shows equivalence to the other common definition of polynomials on projective spaces with the embedding via projection matrices. Just write it out directly, it should work?

% Designs
%TODO check
We now introduce weighted and unweighted $t$-designs, numerical integration formulas (also known as \emph{quadrature} or \emph{cubature rules}) which precisely average polynomials of degree $t$ or less.
\begin{definition} \label{def:weight}
Take $(\Omega,\mu)$ to be a sphere $(S^m,\sigma)$ or projective space $(\Bbb{KP}^n,\rho)$, so we note $\mu(\Omega)=1$.
The pair $(X,\lambda)$ of a finite subset $X\subset\Omega$ and a function $\lambda:X\to\Bbb R_+:=\{r\in\Bbb R\:|\:r>0\}$ is called a \emph{weighted t-design} on $\Omega$ if
\[\sum_{x\in X}\lambda(x)f(x)=\int_{\Omega}f\,d\mu\quad\text{\it for all}\quad f\in P_t(\Omega).\]
\end{definition}
We may call $(X,\lambda)$ a \emph{weighted spherical $t$-design} if $\Omega$ is a sphere and a \emph{weighted $\Bbb K$-projective $t$-design} if $\Omega$ is a $\Bbb K$-projective space.
Additionally, $X$ is called an \emph{unweighted $t$-design} (or simply a \emph{$t$-design}) if $\lambda$ is the constant function $|X|^{-1}$.

% Examples of designs
%TODO make points bigger? Unsure
\begin{figure}
\begin{center}
\includegraphics[width=.8\textwidth]
{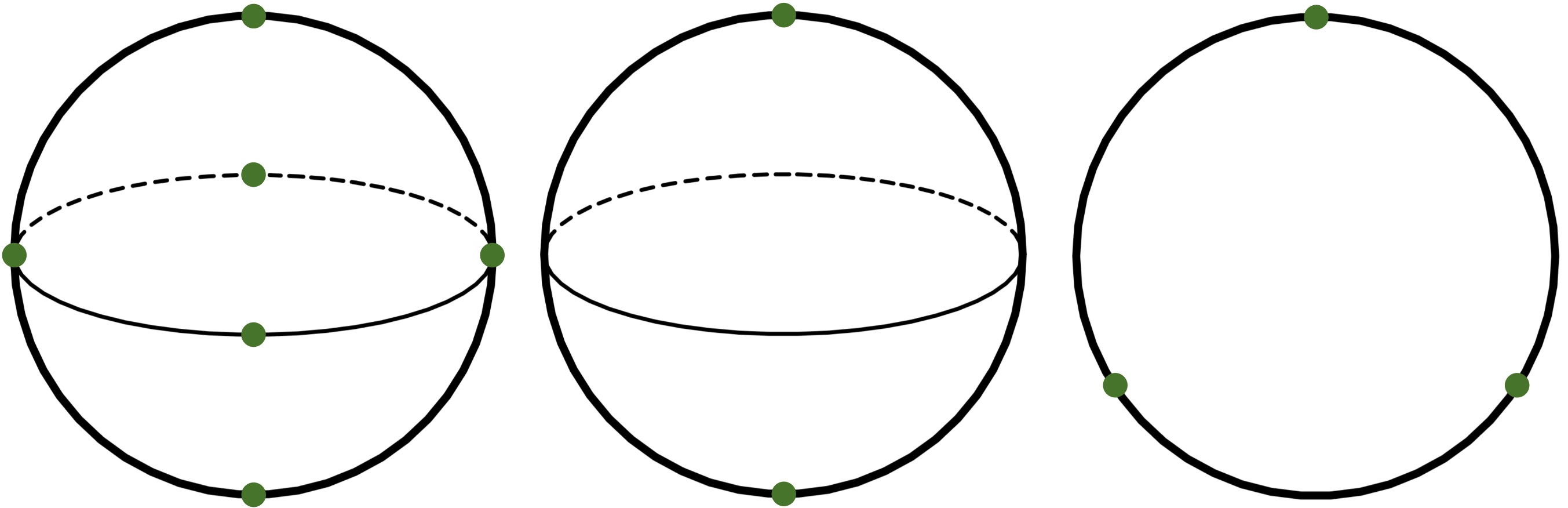}
\caption{\label{fig:designex} Left: vertices of an octahedron, a 3-design on $S^2$. Center: antipodal points, a 1-design on $S^2$. Right: vertices of a triangle, a 2-design on $S^1$.}
\end{center}
\end{figure}

% Other definitions of projective designs
%TODO check
Literature discussing projective designs often defines these objects as finite subsets of projective spaces on which the averages of certain Jacobi polynomials vanish. It was proven by Lyubich \cite[Proposition 2.2]{Lyubich09} that the projective settings of Definition \ref{def:weight} are equivalent to the definition more common in the literature. We note that the work of Lyubich has a typo: the word ``tight'' in their proposition should be omitted.

\section{Validity of the constructions}
\label{sec:proofs}

% Validity exposition
%TODO check
In this section, we prove Theorem \ref{thm:bigthm}.
We first verify Theorem \ref{thm:bigthm} when $\Pi$ is a projective map, then show in Subsection \ref{sub:hopf} that the $n=1$ cases of the projective settings give rise to the results of Theorem \ref{thm:bigthm} when $\Pi$ is a generalized Hopf map.

% Proof of the projective settings of the main theorem
%TODO check
\begin{proof}[Proof of the projective settings of Theorem \ref{thm:bigthm}]
Taking $\Bbb K$, $k$, $n$ and $d$ as in \eqref{eq:Kknd} along with $t\in\Bbb N$, we consider weighted 0-designs $(Y,\lambda_Y)$ on $\Bbb{KP}^n$ and $(Z_y,\lambda_y)$ on $S^k$ alongside base points $z_y\in y$ for each $y\in Y$. We then define $(X,\lambda)$ as in \eqref{eq:X}.

% Setup and the secondary (downwards) construction
%TODO check
Pick $g\in P_{\left\lfloor t/2\right\rfloor}(\Bbb{KP}^n)$.
Noting that $(\Pi_\Bbb K^*g)(z_yz)=g(y)$ for any $z\in Z_y$ and that $(Z_y,\lambda_y)$ is a 0-design on $S^k$ for each $y\in Y$, we have
\begin{equation}\label{eq:pig1}
\begin{split}
\sum\limits_{y\in Y}\sum_{z\in Z_y}\lambda_Y(y)\lambda_y(z)(\Pi_\Bbb K^*g)(z_yz)&=\sum\limits_{y\in Y}\sum_{z\in Z_y}\lambda_Y(y)\lambda_y(z)g(y) \\
&=\sum\limits_{y\in Y}\lambda_Y(y)g(y).
\end{split}
\end{equation}
Additionally, we may observe from the definition \eqref{eq:IK} of $I_\Bbb K$ that
\begin{equation}\label{eq:pig2}
g=(I_\Bbb K\circ\Pi_\Bbb K^*)g.
\end{equation}
Combining \eqref{eq:pig1} and \eqref{eq:pig2} with Definition \ref{def:weight}, we see that $(Y,\lambda_Y)$ is a weighted $\lfloor t/2\rfloor$-design on $\Bbb{KP}^n$ exactly when
\begin{equation}\label{eq:Yisadesign}
\sum\limits_{y\in Y}\sum_{z\in Z_y}\lambda_Y(y)\lambda_y(z)f(z_yz)=\int_{\Bbb{KP}^n}I_\Bbb Kf\,d\rho
\end{equation}
for all $f\in\Pi_\Bbb K^*\bigl(P_{\left\lfloor t/2\right\rfloor}(\Bbb{KP}^n)\bigr)$. Now, we also see from Definition \ref{def:weight} combined with the result
\[\int_{\Bbb{KP}^n}I_\Bbb Kf\,d\rho=\int_{S^d}f\,d\sigma\quad\text{\it{for}}\quad f\in L^1(S^d)\]
of Lemma $\ref{lem:int}$ that \eqref{eq:Yisadesign} being satisfied for $f\in P_t(S^d)$ is exactly the condition that $(X,\lambda)$ is a weighted $t$-design on $S^d$. Therefore, since we can see from the definition \eqref{eq:KPnpoly} of polynomials on projective spaces that
\[\Pi_\Bbb K^*\Bigl(P_{\left\lfloor t/2\right\rfloor}(\Bbb{KP}^n)\Bigr)\subseteq P_t(S^d),\]
we have shown that $(Y,\lambda_Y)$ is a weighted $\lfloor t/2\rfloor$-design on $\Bbb{KP}^n$ if $(X,\lambda)$ is a weighted $t$-design on $S^d$.

%The main (upwards) construction
%TODO check
Now, say $(Z_y,\lambda_y)$ is a weighted $t$-design on $S^k$ for each $y\in Y$.
Picking $f\in P_t(S^d)$, $\eqref{eq:zetawpoly}$ in Lemma $\ref{lem:poly}$ shows that $f(z_{y}\zeta)\in P_t(S^k)$
for $y\in Y$. So, applying a change of variables $z_y\zeta\mapsto\zeta$ using \eqref{eq:zetawpush} followed by the definition \eqref{eq:IK} of $I_\Bbb K$, we have that
\[\sum_{z\in Z_y}\lambda_y(z)f(z_yz)=\int_{S^k}f(z_{y}
\zeta)\,d\sigma(\zeta)=\int_{y}f(\zeta)\,d\sigma(\zeta)=(I_{\Bbb K} f)(y)\]
for $y\in Y$, which shows that 
\begin{equation}\label{eq:sumequals}
\sum\limits_{y\in Y}\sum_{z\in Z_y}\lambda_Y(y)\lambda_y(z)f(z_yz)=\sum_{y\in Y}\lambda_Y(y)(I_{\Bbb K} f)(y).
\end{equation}
From $\eqref{eq:IKpoly}$ in Lemma $\ref{lem:poly}$, we see that $I_\Bbb Kf\in P_{\left\lfloor t/2\right\rfloor}(\Bbb{KP}^n)$, so combining \eqref{eq:sumequals} with \eqref{eq:pig1} for $g=I_\Bbb Kf$ gives us that
\begin{equation}\label{eq:bigsum}
\sum\limits_{y\in Y}\sum_{z\in Z_y}\lambda_Y(y)\lambda_y(z)f(z_yz)=\sum\limits_{y\in Y}\sum_{z\in Z_y}\lambda_Y(y)\lambda_y(z)((\Pi_\Bbb K^*\circ I_\Bbb K)f)(z_yz).
\end{equation}
Additionally, we observe from \eqref{eq:pig2} that
\begin{equation}\label{eq:bigint}
\int_{\Bbb{KP}^n}(I_\Bbb K\circ\Pi_\Bbb K^*\circ I_\Bbb K)f\,d\rho=\int_{\Bbb{KP}^n}I_\Bbb Kf\,d\rho.
\end{equation}
From \eqref{eq:bigsum} and \eqref{eq:bigint}, we see that \eqref{eq:Yisadesign} is satisfied for $f\in P_t(S^d)$ if it is satisfied for 
\[(\Pi_\Bbb K^*\circ I_\Bbb K)f\in\Pi_\Bbb K^*\Bigl(P_{\left\lfloor t/2\right\rfloor}(\Bbb{KP}^n)\Bigr),\]
so $(X,\lambda)$ is a weighted $t$-design on $S^d$ if $(Y,\lambda_Y)$ is a weighted $\lfloor t/2\rfloor$-design on $\Bbb{KP}^n$.
\end{proof}

% Inspiration for proof
%TODO check
Elements of the proof of the projective settings of Theorem \ref{thm:bigthm} were inspired by work of Okuda \cite[Lemma 3.9]{Okuda15} verifying the $\Bbb C$-Hopf setting of Theorem \ref{thm:bigthm}.
% Local trivializations
%TODO check
We now prove facts cited in this proof. With $\Bbb K$, $k$, $n$, and $d$ as in \eqref{eq:Kknd}, consider the left multiplication by a base point $z_w\in w$ map
\begin{equation} \label{eq:zetach}
\zeta_w:S^k\to w,\quad\zeta\mapsto z_w\zeta
\end{equation}
we define for each $w\in\Bbb{KP}^n$. Taking
\begin{equation}\label{eq:uvf}
U:=\{[\omega]\in\Bbb{KP}^n\:|\:\omega_{n+1}\neq0\}
\end{equation}
along with
\begin{equation}\label{eq:vf}
V:=\Pi_\Bbb K^{-1}(U)=\{\omega\in S^d\:|\:\omega_{n+1}\neq0\}
\end{equation}
and respectively equipping these spaces with the restrictions of the measures $\rho$ and $\sigma$ discussed in Subsection \ref{sub:prelims},
we consider a local trivialization 
\begin{equation}\label{eq:psiFdef}
\zeta_U:U\times S^k\to V,\quad([\omega],\zeta)\mapsto\frac{(\omega\overline\omega_{n+1})\zeta}{|\omega_{n+1}|}
\end{equation}
of the $\Bbb K$-projective map, where we equip $U\times S^k$ with the product measure $\mu:=\rho\times\sigma$. %We see that $\zeta_U$ is well-defined from the definition \eqref{eq:KPn} of an element $[\omega]\in\Bbb{KP}^n$. 
%Note that $\zeta_{w}=\zeta_U(w,\cdot)$ for $w\in U$ if we take $z_{w}=\zeta_U(w,1)\in w$.

% How \zeta_w and \zeta_K push forward measures
%TODO check
Observe that $\zeta_w$ pushes forward the uniform measure on $S^k$ to the uniform measure on $w\cong S^k$; i.e.
\begin{equation}\label{eq:zetawpush}%TODO old name chjacobian2
(\zeta_w)_*\sigma=\sigma\quad\text{\it for any}\quad w\in\Bbb{KP}^n.
\end{equation}
We now discuss how $\zeta_U$ pushes forward the product measure $\mu$ on $U\times S^k$ to the restriction of the uniform measure $\sigma$ on $S^d$ to $V$:
\begin{lemma}\label{lem:push}
We have
\begin{equation}\label{eq:zetaKpush}%TODO old name chjacobian
%TODO make sure nothing later says this is a part of the lemma
(\zeta_U)_*\mu=\sigma.%TODO make sure I write things in terms of restrictions in the proof and later on
\end{equation}
\end{lemma}

% Proof of pushforward measure lemma
%TODO check
\begin{proof}
\eqref{eq:zetaKpush} will be satisfied exactly when 
\begin{equation}\label{eq:zetaK-1push}
(\zeta_U^{-1})_*(\sigma)=\mu=\rho\times\sigma.
\end{equation} 
We can see that
\[\zeta_U^{-1}:V\to U\times S^k,\quad \omega\mapsto\left([\omega],\frac{\omega_{n+1}}{|\omega_{n+1}|}\right)=:(\zeta_1^{-1},\zeta_2^{-1})(\omega),\]
so we have 
\[(\zeta_1^{-1})_*\sigma=\rho,\quad(\zeta_2^{-1})_*(\sigma^d)=\sigma^k,\]
as both pushforward measures are uniform by symmetry %TODO enough justification
and normalized such that 
\[((\zeta_1^{-1})_*\sigma)(U)=1=\rho(U),\quad ((\zeta_2^{-1})_*\sigma)(S^k)=1=\sigma(S^k).\]
This exactly shows that \eqref{eq:zetaK-1push} is satisfied.
\end{proof}

% Integration lemma intro
%TODO check
Similarly to as was done by Okuda \cite[Lemma 4.2]{Okuda15}, we use \eqref{eq:zetawpush} and \eqref{eq:zetaKpush} alongside Fubini's theorem \cite{Fubini07} to show that, with $\Bbb K$, $n$, and $d$ as in \eqref{eq:Kknd}, taking the average of any Lebesgue integrable function $f$ over $S^d$ provides the same result as taking the average on $\Bbb{KP}^n$ of the function $(I_{\Bbb K}f)(w)$ which averages $f$ over each $\Bbb K$-projective fiber $w$. For such $f$ and $w$, recalling as in Subsection \ref{sub:prelims} that we normalize the measure $\sigma$ on $w$ such that $\sigma(w)=1$, we have
\begin{equation} \label{eq:IK}%TODO previous label Ich
(I_{\Bbb K}f)(w)=\int_{w}f\,d\sigma.
\end{equation}

% Polynomial averaging figure
%TODO check
\begin{figure}
\begin{center}
\includegraphics[width=.8\textwidth]
{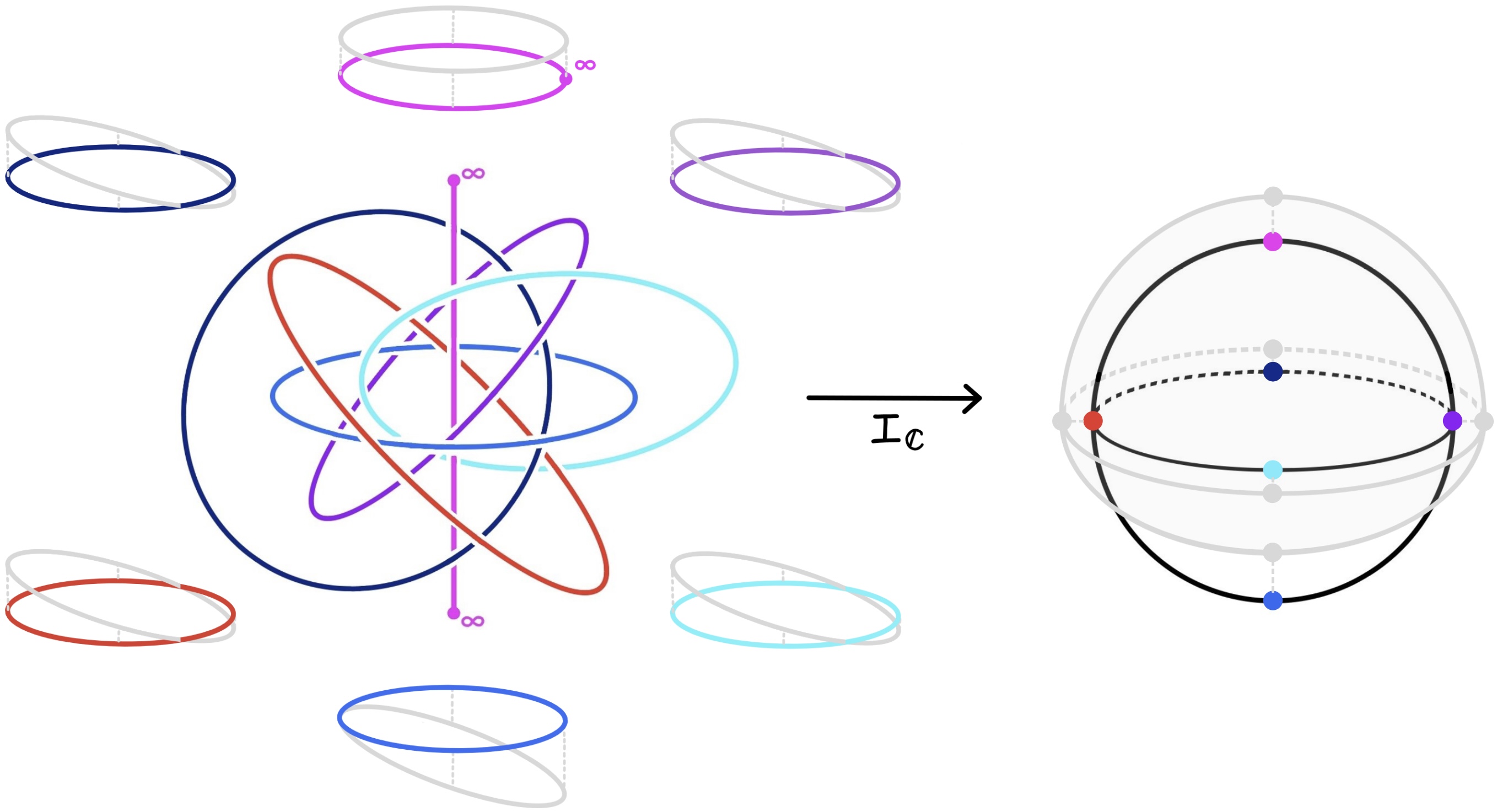}
\caption{\label{fig:PolyExample‎‎}
Taking coordinates $(a,b=b_1+ib_2)\in\Bbb C^2$ on $S^3\cong\Bbb{R}^3\cup\{\infty\}$ and $(\xi,\eta)\in\Bbb R\times\Bbb C$ on $\Bbb{CP}^1\cong S^2$, restrictions of $f=|a|^2-|b|^2+2|a|^2|b|^2-b_1$ to projective fibers are pictured in gray on the left and $(I_\Bbb Cf)(\xi,\eta)=\xi+\frac12|\eta|^2$ is visualized in gray on the right, with points on the right denoted by the same color as their corresponding projective fibers on the left.} \end{center}
\end{figure}

% Integration lemma
%TODO check
\begin{lemma}\label{lem:int}
For $f\in L^1(S^d)$, we have
\begin{equation}\label{eq:intequal}
\int_{\Bbb{KP}^n}I_\Bbb Kf\,d\rho=\int_{S^d}f\,d\sigma.
\end{equation}
\end{lemma}

% Proof of the lemma
%TODO check
\begin{proof}
Pick $f\in L^1(S^d)$ and recall the definitions \eqref{eq:uvf} and \eqref{eq:vf} of $U$ and $V$ respectively. We see using that $S^d\setminus V$ has measure zero in $S^d$ followed by a change of variables using \eqref{eq:zetaKpush} in Lemma \ref{lem:push} that
\begin{equation}\label{eq:mzerovar1}
\int_{S^d}f\,d\sigma=\int_{V}f\,d\sigma=\int_{U\times S^k}\zeta_U^*f\,d\mu.
\end{equation}
Since $f\in L^1(S^d)$ and $\mu=\rho\times\sigma$ is the product measure on $U\times S^k$, Fubini's theorem \cite{Fubini07} followed by the fact that, taking $z_{w}=\zeta_U(w,1)\in w$ for each $w\in\Bbb{KP}^n$, $\zeta_{w}=\zeta_U(w,\cdot)$ shows that
\begin{equation}\label{eq:fubini}
\int_{U\times S^k}\zeta_U^*f\,d\mu=\int_{U}\int_{S^k}\zeta_U^*f\,d\sigma\,d\rho=\int_{U}\int_{S^k}\zeta_w^*f\,d\sigma\,d\rho(w).
\end{equation}
Changing variables using \eqref{eq:zetawpush} and applying the definition \eqref{eq:IK} of $I_\Bbb K$ alongside the fact that $\Bbb{KP}^n\setminus U$ has measure zero in $\Bbb{KP}^n$, we have
\begin{equation}\label{eq:mzerovar2}
\int_{U}\int_{S^k}\zeta_w^*f\,d\sigma\,d\rho(w)=\int_{U}\int_{w}f\,d\sigma\,d\rho=\int_{\Bbb{KP}^n}(I_{\Bbb K} f)\,d\rho.
\end{equation}
Combining \eqref{eq:mzerovar1}, \eqref{eq:fubini}, and \eqref{eq:mzerovar2} then provides us with the desired result \eqref{eq:intequal}.
\end{proof}

% How I_K and \zeta_w relate polynomials
%TODO check
Again consider $\Bbb K$, $k$, $n$, and $d$ as in \eqref{eq:Kknd} along with $t\in\Bbb N$. We now discuss how $I_{\Bbb K}$ as in \eqref{eq:IK} and the pullback operator $\zeta_w^*$ induced by $\zeta_w$ as in \eqref{eq:zetach} relate polynomials. %TODO maybe put a remark here about how I_K circ Pi_K* is the identity (and the other way around on Pi_k^*(P))
\begin{lemma}\label{lem:poly}%TODO formerly lem:chpoly
We have
\begin{gather}%TODO formerly eq:chzetapoly, eq:chIpoly
\zeta_w^*\bigl(P_t(S^d)\bigr)=P_t(S^k)\quad\text{\it for any}\quad w\in\Bbb{KP}^n,\label{eq:zetawpoly} \\
I_{\Bbb K}\bigl(P_t(S^d)\bigr)=P_{\left\lfloor t/2\right\rfloor}(\Bbb{KP}^n). \label{eq:IKpoly}
\end{gather}
\end{lemma}

%Proof of the lemma
%TODO check
\begin{proof}
Fix $w\in\Bbb{KP}^n$ and consider 
\[W:=\{\lambda\omega\:|\:\lambda\in\Bbb R,\:\omega\in w\}\cong\Bbb R^{k+1}.\]
Extending $\zeta_w:S^k\to w$ to the map $\zeta_w:\Bbb R^{k+1}\cong\Bbb K\to W$ taking $a\mapsto z_wa$, the definition of polynomials on $\Bbb R^{m+1}$ in Subsection \ref{sub:prelims} makes it clear that
\[\zeta_w^*\bigl(P_t(\Bbb R^{d+1})\bigr)=P_t(\Bbb R^{k+1}).\]
The definition \eqref{eq:Sdpoly} of polynomials on spheres then shows \eqref{eq:zetawpoly}.%TODO make sure this proof is reasonable

%TODO check
Note that the definition \eqref{eq:IK} of $I_\Bbb K$ shows that $I_\Bbb K\circ \Pi_\Bbb K^*$ is the identity. Combining the definition of $P_{\left\lfloor t/2\right\rfloor}(\Bbb{KP}^n)$ with this fact then shows that
\[P_{\left\lfloor t/2\right\rfloor}(\Bbb{KP}^n)=I_{\Bbb K}(\Pi_\Bbb K^*(P_{\left\lfloor t/2\right\rfloor}(\Bbb{KP}^n)))\subseteq I_{\Bbb K}(P_{t}(S^d)).\]
To prove \eqref{eq:IKpoly}, we must then only show that $I_{\Bbb K}(P_{t}(S^d))\subseteq P_{\left\lfloor t/2\right\rfloor}(\Bbb{KP}^n)$, and the definition \eqref{eq:KPnpoly} of polynomials on $\Bbb{KP}^n$ shows that this will be true if 
\begin{equation}\label{eq:Pi*IKfanevenpoly}
\Pi_\Bbb K^*(I_\Bbb Kf)\in P_{2\left\lfloor t/2\right\rfloor}(S^d)
\end{equation}
for all $f\in P_t(S^d)$. Pick such $f$ to be a monomial of degree $s$. We can observe from the description of multiplication of elements of $\Bbb K$ in Subsection \ref{sub:prelims} that for any $\omega\in S^d\subset\Bbb K^{n+1}$ and $\zeta\in S^k\subset\Bbb K$, %TODO explain better - maybe write the fact above that multiplication by an element of Sk is a degree 2 poly?
\[f(\omega \zeta)=\sum_{i=1}^{(k+1)^s}g_i(\omega)h_i(\zeta),\]
where $g_i$ and $h_i$ are monomials of degree $s$ on $S^d$ and $S^k$ respectively for all $i\in\{1,\dots,(k+1)^s\}$. Therefore, applying a change of variables $\zeta\mapsto\omega\zeta$ for any base point $\omega\in w\in\Bbb{KP}^n$ using \eqref{eq:zetawpush}, we see that
\[(I_\Bbb Kf)(w)=\int_wf(\zeta)\,d\sigma(\zeta)=\int_{S^k}f(\omega\zeta)\,d\sigma(\zeta)=\sum_{i=1}^{(k+1)^s}g_i(\omega)\int_{S^k}h_i(\zeta)\,d\sigma(\zeta).\]
Applying this fact, we see
\[(\Pi_\Bbb K^*(I_\Bbb Kf))(\omega)=(I_\Bbb Kf)[\omega]=\sum_{i=1}^{(k+1)^s}g_i(\omega)\int_{S^k}h_i(\zeta)\,d\sigma(\zeta)\in P_s(S^d).\]
Now, Lemma \ref{lem:*zero} shows that $\Pi_\Bbb K^*(I_\Bbb Kf)=0$ if $s=2\left\lfloor t/2\right\rfloor+1$, so \eqref{eq:Pi*IKfanevenpoly} is satisfied, completing the proof.
\end{proof}

%Lemma showing odd degree polynomials map under IK to 0
%TODO check
\begin{lemma}\label{lem:*zero}
For any polynomial $f\in P_{2t+1}(S^d)$ with terms having only odd degree, $I_\Bbb Kf=0$.
\end{lemma}

%Proof of the lemma
%TODO check
\begin{proof}
Say $f\in P_{2t+1}(S^d)$ is as in the statement of the lemma. We assume without loss of generality that $f$ is a monomial of odd degree $s$. If $\Bbb K=\Bbb R$, we have 
\[(I_\Bbb Kf)([\omega])=\frac12(f(\omega)+f(-\omega))=0\] 
for any $\omega\in S^d$. Now, take $\Bbb K\in\{\Bbb C,\Bbb H,\Bbb O\}$, so $k+1=\dim_\Bbb R\Bbb K$ is even and thus $d+1=(k+1)(n+1)$ is even for any $n\in\Bbb N$. The antipodal map on $S^d$ has differential $-I_{d+1}$ (the negative of the identity on $\Bbb R^{d+1}$), so since
\[\det(-I_{d+1})=(-1)^{d+1}=1,\]
we see applying the change of variables $\zeta\mapsto-\zeta$ that
\[(I_\Bbb Kf)(w)=\int_wf(\zeta)\,d\sigma(\zeta)=(-1)^s\int_wf(\zeta)\,d\sigma(\zeta)=-(I_\Bbb Kf)(w)\]
because $f$ is a monomial of odd degree $s$.
So, we must have $I_{\Bbb K}f=0$.
\end{proof}

\subsection{The generalized Hopf settings}
\label{sub:hopf}

% Hopf settings exposition
%TODO check
We show in this subsection that the results of Theorem \ref{thm:bigthm} which use generalized Hopf maps to relate designs are equivalent to the $n=1$ cases of the corresponding projective settings verified in Section \ref{sec:proofs}. To this end, pick $\Bbb K$, $k$, and $d$ as in \eqref{eq:Kknd} and recall the diffeomorphism $h_\Bbb K$ described in \eqref{eq:hK} which identifies $\Bbb{KP}^1$ with $S^{k+1}$.

% Lemma showing that h relates spaces of polynomials and its implications
%TODO check
\begin{lemma} \label{lem:hlem}
We have
\[h_{\Bbb K}^*\bigl(P_t(S^{k+1})\bigr)=P_t(\Bbb{KP}^1).\]
\end{lemma}
As Lemma \ref{lem:hlem} also implies that 
\[P_t(S^{k+1})=(h_{\Bbb K}^{-1})^*(P_t(\Bbb{KP}^1)),\]
the lemma combined with the fact that $(h_{\Bbb K})_*\mu=\sigma$ shows Corollary \ref{rmk:hcorr}.
\begin{corollary}\label{rmk:hcorr}
$(Y,\lambda_Y)$ is a weighted $t$-design on $\Bbb{KP}^1$ if and only if $(h_{\Bbb K}(Y),$ $(h_{\Bbb K}^{-1})^*\lambda_Y)$ is a weighted $t$-design on $S^{k+1}$.
\end{corollary}
Applying Corollary \ref{rmk:hcorr}, the $n=1$ case of the $\Bbb K$-projective setting of Theorem \ref{thm:bigthm} proves the corresponding $\Bbb K$-Hopf setting of the theorem.

% Proof of the lemma showing that h relates spaces of polynomials
%TODO check
\begin{proof}[Proof of Lemma \ref{lem:hlem}]
For any $(a,b)\in S^{2k+1}$, we note that 
\[w:=[(a,b)]=\pi_\Bbb K^{-1}(\pi_\Bbb K(a,b)).\]
Therefore, for any $g\in P_t(S^{k+1})$, we have from the definition \eqref{eq:IK} of $I_\Bbb K$ that
\[((I_\Bbb K\circ\pi_\Bbb K^*)g)(w)=\int_{\pi_\Bbb K^{-1}(\pi_\Bbb K(a,b))}\pi_\Bbb K^*g\,d\sigma=(\pi_\Bbb K^*g)(a,b).\]
We can therefore see from definitions \eqref{eq:hopf} of $\pi_\Bbb K$ and \eqref{eq:hK} of $h_\Bbb K$ that
\begin{equation}\label{eq:IKopiK*=hK*}
I_\Bbb K\circ\pi_\Bbb K^*=h_\Bbb K^*.
\end{equation}
We also can observe from the description of multiplication of elements of $\Bbb K$ in Subsection \ref{sub:prelims} %TODO explain better - maybe write the fact above that multiplication by an element of Sk is a degree 2 poly?
that
\begin{equation}\label{eq:subset}
\pi_\Bbb K^*(P_t(S^{k+1}))\subseteq P_{2t}(S^{2k+1}).
\end{equation}
Combining \eqref{eq:IKopiK*=hK*} with \eqref{eq:subset} and \eqref{eq:IKpoly} in Lemma \ref{lem:poly}, we see that
\[h_\Bbb K^*(P_t(S^{k+1}))\subseteq I_\Bbb K(P_{2t}(S^{2k+1}))=P_t(\Bbb{KP}^1).\]
So, to complete the proof, we must only show that
\begin{equation}\label{eq:revinc}
h_\Bbb K^*\bigl(P_t(S^{k+1})\bigr)\supseteq P_t(\Bbb{KP}^1),
\end{equation}
which will in turn be true if
\begin{equation}\label{eq:pi*=p*}
\pi_\Bbb K^*(P_t(S^{k+1}))\supseteq \Pi_\Bbb K^*(P_t(\Bbb{KP}^1))
\end{equation}
since applying $I_\Bbb K$ to each side of \eqref{eq:pi*=p*} gives \eqref{eq:revinc} by \eqref{eq:IKopiK*=hK*} and the fact we may observe from the definition \eqref{eq:IK} of $I_\Bbb K$ that $I_\Bbb K\circ \Pi_\Bbb K^*$ is the identity map. Fix $f\in \Pi_\Bbb K^*(P_t(\Bbb{KP}^1))$ and note that $f\in P_{2t}(S^{2k+1})$ by the definition \eqref{eq:KPnpoly} of $P_t(\Bbb{KP}^1)$, so we may consider $f_{\alpha,\beta}\in P_{2t}(S^{2k+1})$ and constants $c_{\alpha,\beta}$ for $\alpha,\beta\in\Bbb N$ satisfying $\alpha+\beta\leq 2t$ such that 
\[f=\sum_{\alpha+\beta\leq2t}f_{\alpha,\beta},\quad |f_{\alpha,\beta}(a,b)|=c_{\alpha,\beta}|a|^\alpha|b|^\beta.\]
Again since $I_\Bbb K\circ \Pi_\Bbb K^*$ is the identity map, we see that $\Pi_\Bbb K^*\circ I_\Bbb K$ is the identity on $\Pi_\Bbb K^*(P_t(\Bbb{KP}^1))$. Also note that this operator is linear. Defining 
\[\widetilde f_{\alpha,\beta}:=(\Pi_\Bbb K^*\circ I_\Bbb K)f_{\alpha,\beta},\]
we therefore have that
\begin{equation}\label{eq:tildesum}
f=(\Pi_\Bbb K^*\circ I_\Bbb K)f=\sum_{\alpha+\beta\leq2t}\widetilde f_{\alpha,\beta}.
\end{equation}
Now, for any $\alpha$ and $\beta$ as above, we see from Lemma \ref{lem:estimate} that there exists $\widetilde c_{\alpha,\beta}\in[0,c]$ satisfying
\[|\widetilde f_{\alpha,\beta}(a,b)|=\widetilde c_{\alpha,\beta}|a|^{\alpha}|b|^{\beta}.\]
Assume without loss of generality that $\alpha\geq \beta$.
We may observe from the definition \eqref{eq:proj} of $\Pi_\Bbb K$ that
\[\widetilde f_{\alpha,\beta}(a,b)=\widetilde f_{\alpha,\beta}\left(|a|,\frac{b\overline a}{|a|}\right)=|a|^{\alpha-\beta}\widetilde f_{\alpha,\beta}(1,0,\dots,0,b\overline a).\]
Lemma \ref{lem:*zero} then guarantees that $\widetilde f_{\alpha,\beta}$ must have terms of only even degree, %TODO saying even but 0 is fine
so we must have $\alpha-\beta=2\gamma_{\alpha,\beta}$ for some $\gamma_{\alpha,\beta}\in\Bbb N$. Therefore, 
\begin{equation*}
\begin{split}
\widetilde f_{\alpha,\beta}(a,b)&=|a|^{2\gamma_{\alpha,\beta}}\widetilde f_{\alpha,\beta}(1,0,\dots,0,b\overline a) \\
&=\left(\frac12(|a|^2-|b|^2)+\frac12\right)^{\gamma_{\alpha,\beta}} \widetilde f_{\alpha,\beta}(1,0,\dots,0,\overline{a\overline b}).
\end{split}
\end{equation*}
We thus have $\widetilde f_{\alpha,\beta}\in\pi_\Bbb K^*(P_{\alpha+\beta}(S^{k+1}))$ by the definition \eqref{eq:hopf} of $\pi_\Bbb K$, showing by linearity of $\pi_\Bbb K^*$ that $f\in \pi_\Bbb K^*(P_t(S^{k+1}))$.
\end{proof}

\begin{lemma} \label{lem:estimate}
For $\alpha,\beta\in\Bbb N$, $f\in P_{\alpha+\beta}(S^{2k+1})$, and $\widetilde f:=(\Pi_\Bbb K^*\circ I_\Bbb K)f$, if
\begin{equation}\label{eq:fdegs}
|f(a,b)|=c|a|^{\alpha}|b|^{\beta}
\end{equation}
for some $c\in\Bbb R$, we have
\begin{equation}\label{eq:tildefdegs}
|\widetilde f(a,b)|=\widetilde c|a|^{\alpha}|b|^{\beta}
\end{equation}
for some $\widetilde c\in[0,c]$.
\end{lemma}

\begin{proof}
Consider $f$ as in the statement of the lemma and label $w:=[(a,b)]$. 
For any 
\[z_w=(z_{w,1},z_{w,2})\in w,\]
we may apply a change of variables $\zeta\mapsto z_w\zeta$ using \eqref{eq:zetawpush} to see that
\begin{equation}\label{eq:pif1}
\begin{split}
|\widetilde f(a,b)|&=\left|\int_{w}f(\zeta)\,d\sigma(\zeta)\right|=\left|\int_{S^k}f(z_{w}\zeta)\,d\sigma(\zeta)\right|\leq\int_{S^k}|f(z_{w}\zeta)|\,d\sigma(\zeta).
\end{split}
\end{equation}
We then see from \eqref{eq:fdegs} that
\begin{equation}\label{eq:pif2}
\begin{split}
\int_{S^k}|f(z_{w}\zeta)|\,d\sigma(\zeta)&=\int_{S^k}c|z_{w,1}\zeta|^\alpha|z_{w,2}\zeta|^\beta\,d\sigma(\zeta) \\
&=c|z_{w,1}|^\alpha|z_{w,2}|^\beta\int_{S^k}|\zeta|^{\alpha+\beta}\,d\sigma(\zeta) \\
&=c|z_{w,1}|^\alpha|z_{w,2}|^\beta.
\end{split}
\end{equation}
The definition \eqref{eq:proj} of $\Pi_\Bbb K$ and the fact that $z_w\in w=[(a,b)]$ then demonstrate that $|z_{w,1}|=|a|$ and $|z_{w,2}|=|b|$. Combining this with \eqref{eq:pif1} and \eqref{eq:pif2}, we can see that
\begin{equation}\label{eq:lessthanpoly}
|\widetilde f(a,b)|\leq c|a|^\alpha|b|^\beta.
\end{equation}
Every (nonzero) term of $\widetilde f(a,b)$ must thus contain at least $\alpha$ multiples of $a$ or $\overline a$ and at least $\beta$ multiples of $b$ or $\overline b$, as \eqref{eq:lessthanpoly} would otherwise not be satisfied when $a\to0$ or $b\to0$.
Now, we can observe from the definition \eqref{eq:KPnpoly} of $P_t(\Bbb{KP}^1)$ and \eqref{eq:IKpoly} in Lemma \ref{lem:poly} that 
\[\widetilde f=(\Pi_\Bbb K^*\circ I_\Bbb K)f\in P_{\alpha+\beta}(S^{2k+1}),\]
so the degree of $\widetilde f$ does not exceed $\alpha+\beta$. Therefore, every (nonzero) term of $\widetilde f(a,b)$ must contain exactly $\alpha$ multiples of $a$ or $\overline a$ and exactly $\beta$ multiples of $b$ or $\overline b$. This fact combined with \eqref{eq:lessthanpoly} completes the proof.
\end{proof}

\section{Examples and optimality}
\label{sec:ex}

% N>t-gon vertices are t-designs
%TODO check
We now present examples of $t$-designs on $S^d$ which can be built using the constructions communicated by Theorem \ref{thm:bigthm}, remark on the sizes of such $t$-designs, and discuss the sense in which these constructions are asymptotically optimal as $t\to\infty$. Note that all examples provided using the complex Hopf or complex projective settings could also have been produced using the results of Okuda \cite[Theorem 1.1]{Okuda15} or K\"{o}nig \cite[Corollary 1]{Koning98} and Kuperberg \cite[Theorem 4.1]{Kuperberg06} respectively.

We use the fact that, for $N\geq t+1$, the vertices
\begin{equation}\label{eq:ngon}
V_N:=\bigl\{e^{2m\pi i/N}\:|\:m\in\{0,\dots,N-1\}\bigr\}
\end{equation}
of a regular $N$-gon are a $t$-design on $S^1$ \cite[Example 5.1.4]{Delsarte...77}.
We also note that we pick base points for the fibers of the $\Bbb K$-Hopf maps using the equation
\begin{equation}\label{eq:explicitfiber}
\pi_\Bbb K^{-1}(\xi,\eta)=
\begin{cases}
\Bigl\{\frac1{\sqrt2}\Bigl(\zeta\sqrt{1+\xi},\frac{\overline\eta \zeta}{\sqrt{1+\xi}}\Bigr)\:\Big|\:\zeta\in S^k\Bigr\}, & \xi\neq-1 \\
\{(0,\zeta)\:|\:\zeta\in S^k\}, & \xi=-1
\end{cases}.
\end{equation}

% The antipodal example
We first discuss an example remarked on by Okuda \cite[Example 2.4]{Okuda15}.
%TODO check
\begin{example} \label{ex:okudaex}
The set $A=\{(\pm1,0)\}$ of poles in $S^2\subset\Bbb R\times\Bbb C$ is a $1$-design. We also know that the vertices $V_3$ of an equilateral triangle are a $2$-design on $S^1$. We see from \eqref{eq:explicitfiber} that
\[\pi_\Bbb C^{-1}(1,0)=\{(\zeta,0)\:|\:\zeta\in S^1\},\quad\pi_\Bbb C^{-1}(-1,0)=\{(0,\zeta)\:|\:\zeta\in S^1\},\]
so the $\Bbb C$-Hopf setting of Theorem \ref{thm:bigthm} gives us that the set
\[X:=\bigl\{\bigl(z_1e^{2m\pi i/3},0\bigr),\bigl(0,z_2e^{2m\pi i/3}\bigr)\:\big|\:m\in\{0,1,2\}\bigr\}\subset S^3\subset\Bbb C^2\]
is a $2$-design on $S^3$ for any $z_1,z_2\in S^1$. Specifically, we see that
\[\bigl\{(1,0),\bigl(e^{2\pi i/3},0\bigr),\bigl(e^{4\pi i/3},0\bigr),(0,1),\bigl(0,e^{2\pi i/3}\bigr),\bigl(0,e^{4\pi i/3}\bigr)\bigr\}\]
is a $2$-design on $S^3$.
This is visualized in Figure \ref{fig:HopfExample}.
\end{example}

% Antipodal example figure
%TODO check
\begin{figure}
\begin{center}
\includegraphics[width=.85\textwidth]
{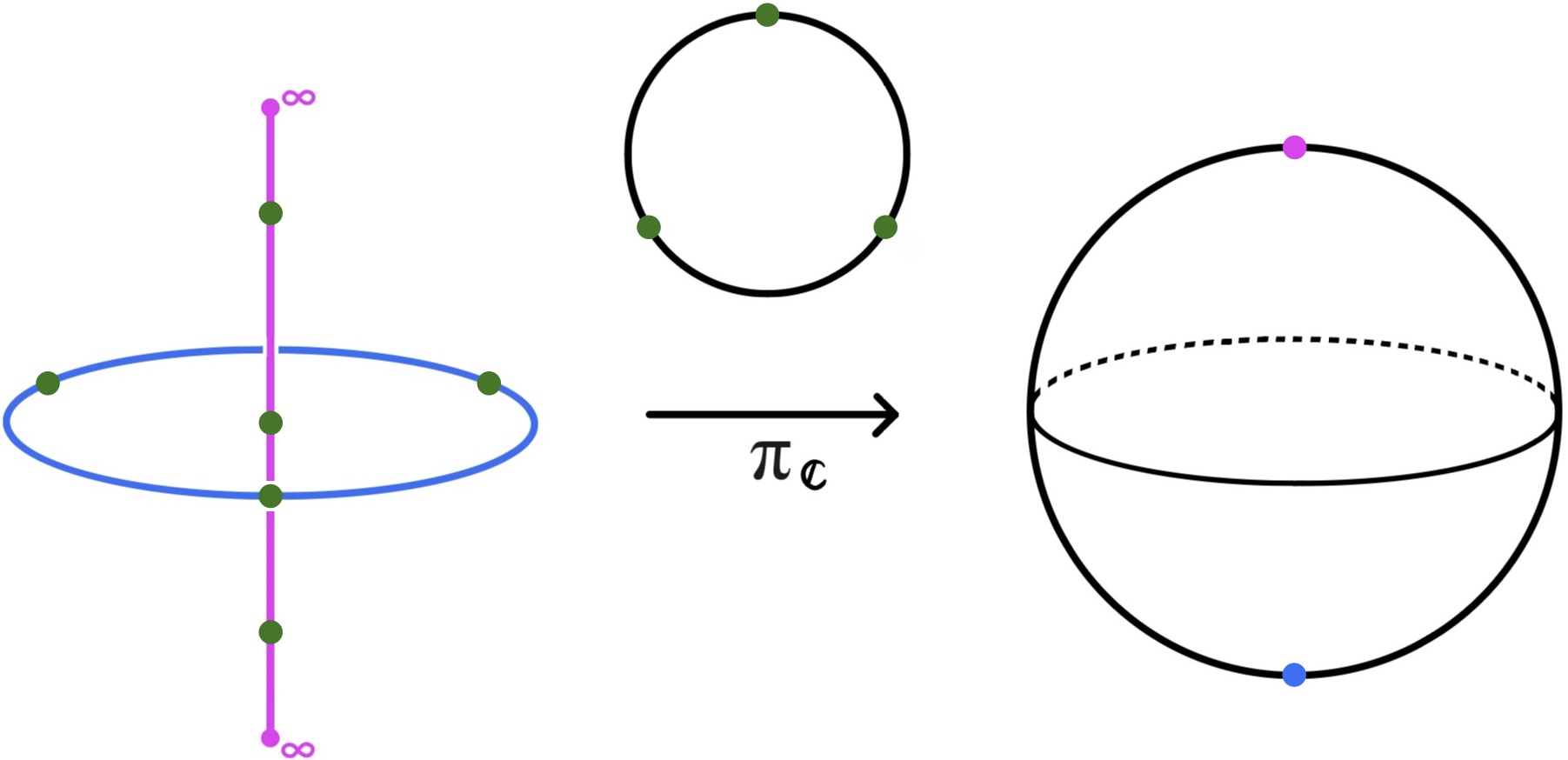}
\caption{\label{fig:HopfExample}A visualization of Example \ref{ex:okudaex}, where we have that $Y=\{(1,0),(-1,0)\}\subset S^2$ is a 1-design and each $ Z_y$, both of which are sets of vertices of equilateral triangles, is a 2-design on $S^1$. Elements of $X$, a 2-design on $S^3\cong\Bbb R^3\cup\{\infty\}$, are shown in green on the left.}
\end{center}
\end{figure}

% The octahedral example
%TODO check
Example \ref{ex:cohnex2} presents a family of 7-designs on $S^3$ discussed by Cohn, Conway, Elkies, and Kumar \cite{Cohn...06}. These authors also discuss the example of the $D_4$ root system, a 5-design on $S^3$ which consists of four sets of points which are each the vertices of a hexagon inscribed in a Hopf fiber. These Hopf fibers have images under the Hopf map which are the vertices of a tetrahedron on $S^2$, a 2-design \cite[Example 2.7]{BannaiBannai09}.
\begin{example} \label{ex:cohnex2}
As discussed by Bannai and Bannai \cite[Example 2.7]{BannaiBannai09}, the vertices
\[O:=\{(1,0),(-1,0),(0,1),(0,-1),(0,e_1),(0,-e_1)\}\]
of a regular octahedron on $S^2\subset\Bbb R\times\Bbb C$ (pictured in Figure \ref{fig:designex}) constitute a $3$-design. We also know that the vertices $V_8$ of a regular $8$-gon are a $7$-design on $S^1$. The $\Bbb C$-Hopf setting of Theorem \ref{thm:bigthm} then gives us that
\begin{equation*}
X:=\bigl\{z_ye^{2m\pi i/8}\:\big|\:y\in O,\:m\in\{0,\dots,7\}\bigr\}
\end{equation*}
is a $7$-design for any base points $z_y\in\pi_\Bbb C^{-1}(y)$ $(y\in O)$. We may then use \eqref{eq:explicitfiber} to pick these base points, giving rise to a concrete family of 7-designs on $S^3$. Figure \ref{fig:CohnExample} visually describes this example.
\end{example}

% 3-design example figure
%TODO check
\begin{figure}
\begin{center}
\includegraphics[width=.8\textwidth]
{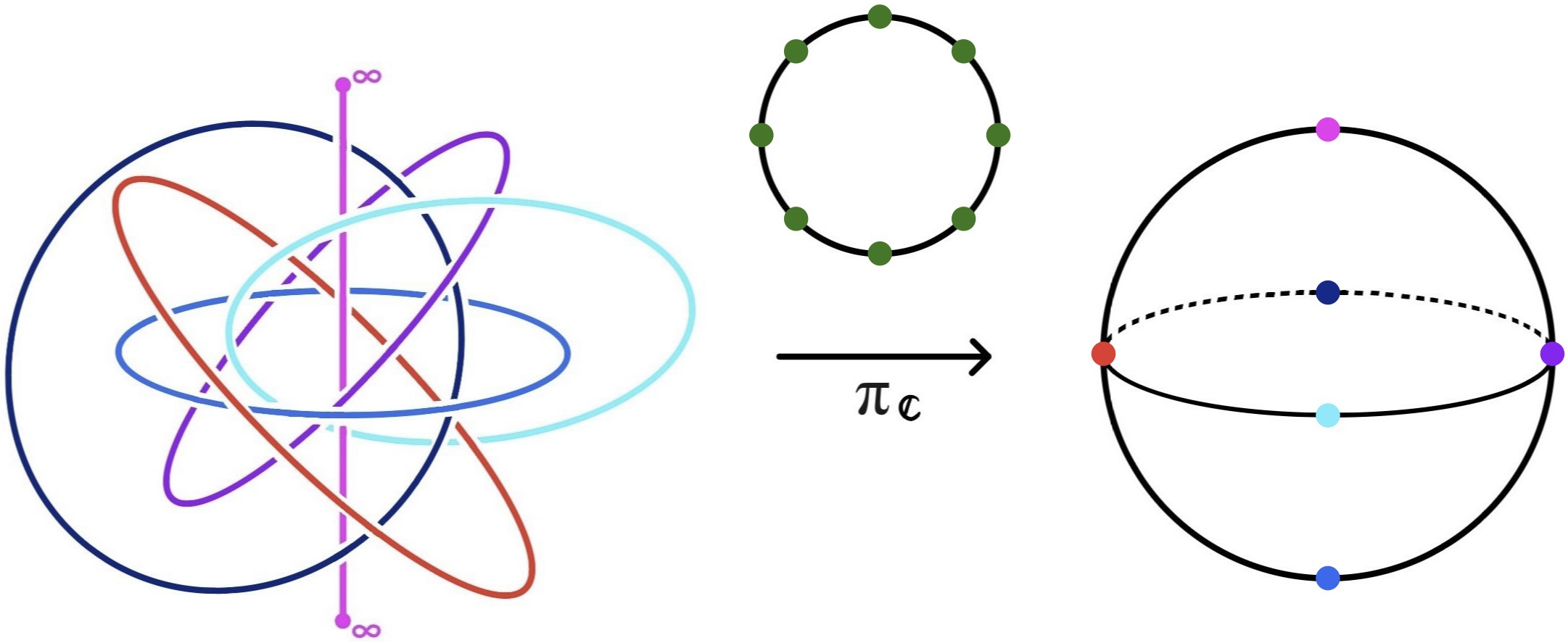}
\caption{\label{fig:CohnExample}The Hopf fibers in $S^3\cong\Bbb R^3\cup\{\infty\}$ corresponding to the vertices of an octahedron on $S^2$. To construct the family of 7-designs described in Example \ref{ex:cohnex2}, distribute points corresponding to the vertices of an octagon on each marked fiber.} \end{center}
\end{figure}

% Table discussing optimality of the Hopf construction
%TODO check
The 24-point 5-design and 48-point 7-design discussed by Cohn, Conway, Elkies, and Kumar \cite{Cohn...06} are putatively optimally small among designs on $S^3$ of their respective strengths. Table \ref{tab:hopf} details when $t$-designs minimally small among those constructed as in the $\Bbb C$-Hopf setting of Theorem \ref{thm:bigthm} have optimal size among all known $t$-designs on $S^3$.
\begin{table}[h]
\centering
\caption{We here have notation as in Theorem \ref{thm:bigthm}, where $Y$ is putatively optimally small among $\lfloor t/2\rfloor$-designs on $S^2$ and we build $X$ taking $Z_y=V_{t+1}$---an optimally small $t$-design on $S^1$ \cite[Example 5.1.4]{Delsarte...77}---for each $y\in Y$. The values $|Y|$ come from work of Hardin and Sloane \cite[Table I]{HardinSloane02} and putatively optimal minimal sizes of $t$-designs on $S^3$ were aggregated by Sloane, Hardin, and Cara \cite[Table 1]{Sloane...03}. To the knowledge of the author, these designs are only known to be optimal for $t=0$ trivially, for $\lfloor t/2\rfloor\in\{1,2,3,5\}$ on $S^2$ and $t\in\{1,2,3\}$ on $S^3$ by the tightness bounds of Delsarte, Goethals, and Seidel \cite[Definition 5.13]{Delsarte...77}, and when $t=11$ on $S^3$ by work of Andreev \cite[Theorem 1]{Andreev00}. While not necessarily optimal, $t$-designs on $S^3$ found by Womersley \cite{Womersley17} may similarly indicate the efficiency of the Hopf construction of Theorem \ref{thm:bigthm} for other values of $t$.} \label{tab:hopf}
\begin{tabular*}{.8\textwidth}{@{\extracolsep{\fill}} *{4}{c}}
\toprule
$t$ & $|Y|$ & $|X|$ & Putatively optimal minimum \\ 
\midrule
0 & 1 & 1 & 1 \\  
1 & 1 & 2 & 2 \\
2 & 2 & 6 & 5 \\  
3 & 2 & 8 & 8 \\
4 & 4 & 20 & 20 \\  
5 & 4 & 24 & 24 \\
6 & 6 & 42 & 42 \\  
7 & 6 & 48 & 48 \\  
8 & 12 & 108 & 96 \\ 
11 & 12 & 144 & 120 \\
\bottomrule
\end{tabular*}
\end{table}

Also note that, as Chen, Frommer, and Lang computationally verified the existence of $t$-designs on $S^2$ with $(t+1)^2$ points for $t\in\{0,...,100\}$ \cite{Chen...11}, we may apply the $\Bbb C$-Hopf setting of Theorem \ref{thm:bigthm} to verify existence of $t$-designs on $S^3$ with $(\lfloor t/2\rfloor+1)^2(t+1)$ points for $t\in\{0,...,201\}$.

% Examples arising from the connection to the Hopf setting
Applying Corollary \ref{rmk:hcorr}, we get examples of the complex projective setting of Theorem \ref{thm:bigthm} from Example \ref{ex:okudaex}, Example \ref{ex:cohnex2}, and the cases discussed in Table \ref{tab:hopf}.
We now give other examples of the projective settings of the theorem. 

% 1-design on projective spaces
%TODO check
With $\Bbb K$, $k$, $n$, and $d$ as in \eqref{eq:Kknd} and
\begin{equation*}\label{eq:bF}
B^d:=\{(1,0,\dots,0),(0,1,0,\dots,0),\dots,(0,\dots,0,1)\}\subset S^d,
\end{equation*}
$\Pi_\Bbb K(B^d)$ is a $1$-design on $\Bbb{KP}^n$ for $\Bbb K\in\{\Bbb R,\Bbb C,\Bbb H,\Bbb O\}$. This can be seen (when $\Bbb K\in\{\Bbb C,\Bbb H\}$) from the fact that a set $\{[w_1],\dots,[w_N]\}$ is a 1-design on $\Bbb{KP}^n$ if
\[\frac1N\sum_{i=1}^Nw_i\overline w_i^t=\frac{1}{n+1}I_{n+1},\]
where $I_{n+1}$ is the identity on $\Bbb K^{n+1}$ \cite[Section 2]{Cohn...16}.

% Higher dimensional real example (used in complex example)
%TODO check
\begin{example}\label{ex:ex1}
As $\{\pm1\}=S^0$ is (trivially) a $t$-design on $S^0$ for any $t$, 
\begin{equation*}\label{eq:X_R}
\begin{split}
X^{n}:&=\bigl\{\pm b\:\big|\:b\in B^n\bigr\}=\{(\pm1,0,\dots,0), (0,\dots,0,\pm1)\}
\end{split}
\end{equation*}
is a $3$-design on $S^{n}$ by the real projective setting of Theorem \ref{thm:bigthm}.
\end{example}
%TODO check
For $\Bbb K\in\{\Bbb C,\Bbb H,\Bbb O\}$, $X^d$ may similarly be built from the 1-design $\Pi_\Bbb K(B^d)\subset\Bbb{KP}^n$ and the 3-design $X^k\subset S^k$ using the $\Bbb K$-projective construction of Theorem \ref{thm:bigthm}. The sets $X^m$ are known to be \emph{tight} 3-designs \cite{BannaiBannai09}, where a $t$-design is called tight if it achieves the lower bound on the size of such a design provided by Delsarte, Goethals, and Seidel \cite[Definition 5.13]{Delsarte...77} in the spherical setting or the analogous bounds of Dunkl \cite{Dunkl79} and of Bannai and Hoggar \cite[Proposition 1.1]{BannaiHoggar85} in the projective setting.

% Tight projective designs
%TODO check
A number of tight projective $t$-designs were presented by Hoggar \cite[Table 3]{Hoggar82} and Levenshtein compiled these and other projective designs into a list \cite[Table 9.2]{Levenshtein92}. Tight designs outlined by Hoggar include a 40-point 3-design on $\Bbb{CP}^3$ \cite[Example 6]{Hoggar82}, a 126-point 3-design on $\Bbb{CP}^5$ \cite[Example 7]{Hoggar82}, and a 165-point 3-design on $\Bbb{HP}^5$ \cite[Example 9]{Hoggar82}. 
Noting that $V_{t+1}$ as in \eqref{eq:ngon} is a tight t-design on $S^1$ for any $t\in\Bbb N_{+}$ \cite[Theorem A]{Hong82} and that the 48-point 7-design on $S^3$ discussed by Cohn, Conway, Elkies, and Kumar \cite{Cohn...06} as an initial example of a design built as in the complex Hopf construction of Theorem \ref{thm:bigthm} is conjectured to be an optimally small 7-design on $S^3$ \cite[Table 1]{Sloane...03}, we get from Theorem \ref{thm:bigthm} a 320-point 7-design on $S^7$, a 1,008-point 7-design on $S^{11}$, and a 7,920-point 7-design on $S^{23}$. While each of these designs is putatively optimally small among 7-designs on their respective spheres built using Theorem \ref{thm:bigthm}, these designs are not themselves tight---in fact, tightness is rare among designs built on $S^d$ as in Theorem \ref{thm:bigthm} for $\Bbb K\in\{\Bbb C,\Bbb H,\Bbb O\}$.

% Asymptotic optimality
%TODO check
Taking $\Bbb K$, $k$, $n$ and $d$ as in \eqref{eq:Kknd} alongside $m\in\Bbb N_+$ and $t\in\Bbb N$, say $N(m,t)$ and $N(\Bbb K,n,t)$ are the minimal numbers of points required to form a $t$-design on $S^m$ and on $\Bbb{KP}^n$ respectively. Work of Bondarenko, Radchenko, and Viazovska \cite[Theorem 1]{Bondarenko...13} combined with the tightness bounds of Delsarte, Goethals, and Seidel \cite[Theorems 5.11, 5.12]{Delsarte...77} shows that $N(m,t)\asymp t^m$ as $t\to\infty$ and work of Etayo, Marzo, and Ortega-Cerd\`a \cite[Theorem 2.2]{Etayo...18} combined with the tightness bounds of Dunkl \cite{Dunkl79} shows that $N(\Bbb K,n,t)\asymp t^{d-k}$ as $t\to\infty$.
This demonstrates that each construction of Theorem \ref{thm:bigthm} is asymptotically optimal as $t\to\infty$, in the sense that an optimally small $t$-design $X_t$ capable of being built on $S^d$ by such a construction will have size 
\[|X_t|=N(\Bbb K,n,\lfloor t/2\rfloor)N(k,t)\asymp t^{d-k}t^k=t^d\quad\text{as}\quad t\to\infty.\]

\section*{Acknowledgements}
\label{sec:thanks}

The author would like to thank Henry Cohn for suggesting this problem and providing helpful insights throughout the research process, as well as the reviewer for helpful comments. The author would also like to thank the School of Science, the Department of Mathematics, and the Office of Graduate Education for the MIT Dean of Science fellowship for support in their doctoral studies and would like to acknowledge NSF grant DMS-2105512 and the Simons Foundation Award \#994330 (Simons Collaboration on New Structures in Low-Dimensional Topology) for their support.

\iffalse
\section*{Statements and declarations}

Funding: Funding for the author in support of their research was provided by the School of Science, the Department of Mathematics, and the Office of Graduate Education through the MIT Dean of Science fellowship, by NSF grant DMS-2105512, and by the Simons Foundation Award \#994330 (Simons Collaboration on New Structures in Low-Dimensional Topology). 

Competing interests: the author has no competing interests to declare. 

Data availability: no data sets were generated in association with this manuscript.
\fi

\bibliography{Bibliography.bib}

@article {Andreev00,
    AUTHOR = {Andreev, N. N.},
     TITLE = {A minimal design of order 11 on the three-dimensional sphere},
   JOURNAL = {Mat. Zametki},
  FJOURNAL = {Matematicheskie Zametki},
    VOLUME = {67},
      YEAR = {2000},
    NUMBER = {4},
     PAGES = {489--497},
      ISSN = {0025-567X,2305-2880},
   MRCLASS = {05B05 (94B30)},
  MRNUMBER = {1769895},
MRREVIEWER = {Victor.\ A.\ Zinoviev},
       DOI = {10.1007/BF02676396},
       URL = {https://doi.org/10.1007/BF02676396},
}

@article{Baez01,
    AUTHOR = {Baez, John C.},
     TITLE = {The octonions},
   JOURNAL = {Bull. Amer. Math. Soc. (N.S.)},
  FJOURNAL = {American Mathematical Society. Bulletin. New Series},
    VOLUME = {39},
      YEAR = {2002},
    NUMBER = {2},
     PAGES = {145--205},
      ISSN = {0273-0979,1088-9485},
   MRCLASS = {17A35 (01A55 01A60 17C40 83A05)},
  MRNUMBER = {1886087},
MRREVIEWER = {Helena\ Albuquerque},
       DOI = {10.1090/S0273-0979-01-00934-X},
       URL = {https://doi.org/10.1090/S0273-0979-01-00934-X},
}

@article {Bajnok92,
    AUTHOR = {Bajnok, Bela},
     TITLE = {Construction of spherical {$t$}-designs},
   JOURNAL = {Geom. Dedicata},
  FJOURNAL = {Geometriae Dedicata},
    VOLUME = {43},
      YEAR = {1992},
    NUMBER = {2},
     PAGES = {167--179},
      ISSN = {0046-5755,1572-9168},
   MRCLASS = {05B30},
  MRNUMBER = {1180648},
MRREVIEWER = {J.\ J.\ Seidel},
       DOI = {10.1007/BF00147866},
       URL = {https://doi.org/10.1007/BF00147866},
}

@article {BannaiBannai09,
    AUTHOR = {Bannai, Eiichi and Bannai, Etsuko},
     TITLE = {A survey on spherical designs and algebraic combinatorics on
              spheres},
   JOURNAL = {European J. Combin.},
  FJOURNAL = {European Journal of Combinatorics},
    VOLUME = {30},
      YEAR = {2009},
    NUMBER = {6},
     PAGES = {1392--1425},
      ISSN = {0195-6698,1095-9971},
   MRCLASS = {05B30 (05E30 52C17)},
  MRNUMBER = {2535394},
MRREVIEWER = {Bertrand\ Meyer},
       DOI = {10.1016/j.ejc.2008.11.007},
       URL = {https://doi.org/10.1016/j.ejc.2008.11.007},
}

@article {BannaiHoggar85,
    AUTHOR = {Bannai, Eiichi and Hoggar, Stuart G.},
     TITLE = {On tight {$t$}-designs in compact symmetric spaces of rank
              one},
   JOURNAL = {Proc. Japan Acad. Ser. A Math. Sci.},
  FJOURNAL = {Japan Academy. Proceedings. Series A. Mathematical Sciences},
    VOLUME = {61},
      YEAR = {1985},
    NUMBER = {3},
     PAGES = {78--82},
      ISSN = {0386-2194},
   MRCLASS = {05B30 (51E05 53C35)},
  MRNUMBER = {796472},
       URL = {http://projecteuclid.org/euclid.pja/1195514811},
}

@article {Bondarenko...13,
    AUTHOR = {Bondarenko, Andriy and Radchenko, Danylo and Viazovska,
              Maryna},
     TITLE = {Optimal asymptotic bounds for spherical designs},
   JOURNAL = {Ann. of Math. (2)},
  FJOURNAL = {Annals of Mathematics. Second Series},
    VOLUME = {178},
      YEAR = {2013},
    NUMBER = {2},
     PAGES = {443--452},
      ISSN = {0003-486X,1939-8980},
   MRCLASS = {41A63 (41A55 52C35 65C10)},
  MRNUMBER = {3071504},
MRREVIEWER = {Hiroshi\ Nozaki},
       DOI = {10.4007/annals.2013.178.2.2},
       URL = {https://doi.org/10.4007/annals.2013.178.2.2},
}

@article{Chen...11,
author = {Chen, Xiaojun and Frommer, Andreas and Lang, Bruno},
year = {2011},
month = {02},
pages = {289-305},
title = {Computational existence proofs for spherical t-designs},
volume = {117},
journal = {Numer. Math.},
fjournal = {Numerische Mathematik},
doi = {10.1007/s00211-010-0332-5}
}

@article {Cohn...06,
    AUTHOR = {Cohn, Henry and Conway, John H. and Elkies, Noam D. and Kumar,
              Abhinav},
     TITLE = {The {$D_4$} root system is not universally optimal},
   JOURNAL = {Experiment. Math.},
  FJOURNAL = {Experimental Mathematics},
    VOLUME = {16},
      YEAR = {2007},
    NUMBER = {3},
     PAGES = {313--320},
      ISSN = {1058-6458,1944-950X},
   MRCLASS = {52C17 (05B40 52A40)},
  MRNUMBER = {2367321},
       URL = {http://projecteuclid.org/euclid.em/1204928532},
}

@article {Cohn...16,
    AUTHOR = {Cohn, Henry and Kumar, Abhinav and Minton, Gregory},
     TITLE = {Optimal simplices and codes in projective spaces},
   JOURNAL = {Geom. Topol.},
  FJOURNAL = {Geometry \& Topology},
    VOLUME = {20},
      YEAR = {2016},
    NUMBER = {3},
     PAGES = {1289--1357},
      ISSN = {1465-3060,1364-0380},
   MRCLASS = {94B25 (49M15 51M16 52C17)},
  MRNUMBER = {3523059},
MRREVIEWER = {Anton\ Vladimirovich\ Shutov},
       DOI = {10.2140/gt.2016.20.1289},
       URL = {https://doi.org/10.2140/gt.2016.20.1289},
}

@article {Delsarte...77,
    AUTHOR = {Delsarte, Philippe and Goethals, Jean-Marie and Seidel, Johan J.},
     TITLE = {Spherical codes and designs},
   JOURNAL = {Geom. Dedicata},
  FJOURNAL = {Geometriae Dedicata},
    VOLUME = {6},
      YEAR = {1977},
    NUMBER = {3},
     PAGES = {363--388},
   MRCLASS = {05B99},
  MRNUMBER = {485471},
MRREVIEWER = {Michel\ Deza},
       DOI = {10.1007/bf03187604},
       URL = {https://doi.org/10.1007/bf03187604},
}

@article {Dunkl79,
    AUTHOR = {Dunkl, Charles F.},
     TITLE = {Discrete quadrature and bounds on {$t$}-designs},
   JOURNAL = {Michigan Math. J.},
  FJOURNAL = {Michigan Mathematical Journal},
    VOLUME = {26},
      YEAR = {1979},
    NUMBER = {1},
     PAGES = {81--102},
      ISSN = {0026-2285,1945-2365},
   MRCLASS = {05B30 (94B25)},
  MRNUMBER = {514963},
MRREVIEWER = {D.\ K.\ Ray-Chaudhuri},
       URL = {http://projecteuclid.org/euclid.mmj/1029002165},
}

@book {Ebbinghaus...91,
    AUTHOR = {Ebbinghaus, H.-D. and Hermes, H. and Hirzebruch, F. and
              Koecher, M. and Mainzer, K. and Neukirch, J. and Prestel, A.
              and Remmert, R.},
     TITLE = {Numbers},
    SERIES = {Graduate Texts in Mathematics},
    VOLUME = {123},
   EDITION = {{G}erman},
      NOTE = {With an introduction by K. Lamotke,
              Translation edited and with a preface by J. H. Ewing,
              Readings in Mathematics},
 PUBLISHER = {Springer-Verlag, New York},
      YEAR = {1990},
     PAGES = {xviii+395},
      ISBN = {0-387-97202-1},
   MRCLASS = {00A05 (01A05 11-03)},
  MRNUMBER = {1066206},
       DOI = {10.1007/978-1-4612-1005-4},
       URL = {https://doi.org/10.1007/978-1-4612-1005-4},
}

@article {Etayo...18,
    AUTHOR = {Etayo, Uju\'{e} and Marzo, Jordi and Ortega-Cerd\`a, Joaquim},
     TITLE = {Asymptotically optimal designs on compact algebraic manifolds},
   JOURNAL = {Monatsh. Math.},
  FJOURNAL = {Monatshefte f\"{u}r Mathematik},
    VOLUME = {186},
      YEAR = {2018},
    NUMBER = {2},
     PAGES = {235--248},
      ISSN = {0026-9255,1436-5081},
   MRCLASS = {41A55 (41A63 52C35 65D30)},
  MRNUMBER = {3808652},
       DOI = {10.1007/s00605-018-1174-y},
       URL = {https://doi.org/10.1007/s00605-018-1174-y},
}

@article{Fubini07,
    Author = {Fubini, Guido},
    Title = {Sugli integrali multipli},
    FJournal = {Accademia dei Lincei, Rendiconti, V. Serie},
    Journal = {Rom. Acc. L. Rend. (5)},
    ISSN = {0001-4435},
    Volume = {16},
    Number = {1},
    Pages = {608--614},
    Year = {1907},
    Language = {Italian},
    zbMATH = {2643959},
    JFM = {38.0343.02}
}

@article {HardinSloane02,
    AUTHOR = {Hardin, Ronald H. and Sloane, Neil J. A.},
     TITLE = {Mc{L}aren's improved snub cube and other new spherical designs
              in three dimensions},
   JOURNAL = {Discrete Comput. Geom.},
  FJOURNAL = {Discrete \& Computational Geometry. An International Journal
              of Mathematics and Computer Science},
    VOLUME = {15},
      YEAR = {1996},
    NUMBER = {4},
     PAGES = {429--441},
      ISSN = {0179-5376,1432-0444},
   MRCLASS = {52B11 (05B30)},
  MRNUMBER = {1384885},
MRREVIEWER = {J.\ J.\ Seidel},
       DOI = {10.1007/BF02711518},
       URL = {https://doi.org/10.1007/BF02711518},
}

@article {Hoggar82,
    AUTHOR = {Hoggar, Stuart G.},
     TITLE = {{$t$}-designs in projective spaces},
   JOURNAL = {European J. Combin.},
  FJOURNAL = {European Journal of Combinatorics},
    VOLUME = {3},
      YEAR = {1982},
    NUMBER = {3},
     PAGES = {233--254},
      ISSN = {0195-6698,1095-9971},
   MRCLASS = {05B30 (62K99)},
  MRNUMBER = {679208},
       DOI = {10.1016/S0195-6698(82)80035-8},
       URL = {https://doi.org/10.1016/S0195-6698(82)80035-8},
}

@article {Hong82,
    AUTHOR = {Hong, Yiming},
     TITLE = {On spherical {$t$}-designs in {${\Bbb R}\sp{2}$}},
   JOURNAL = {European J. Combin.},
  FJOURNAL = {European Journal of Combinatorics},
    VOLUME = {3},
      YEAR = {1982},
    NUMBER = {3},
     PAGES = {255--258},
      ISSN = {0195-6698,1095-9971},
   MRCLASS = {05B30},
  MRNUMBER = {679209},
MRREVIEWER = {J.\ J.\ Seidel},
       DOI = {10.1016/S0195-6698(82)80036-X},
       URL = {https://doi.org/10.1016/S0195-6698(82)80036-X},
}

@incollection {Levenshtein92,
    AUTHOR = {Levenshte{\u i}n, V. I.},
     TITLE = {Designs as maximum codes in polynomial metric spaces},
      NOTE = {Interactions between algebra and combinatorics},
   JOURNAL = {Acta Appl. Math.},
  FJOURNAL = {Acta Applicandae Mathematicae},
    VOLUME = {29},
      YEAR = {1992},
    NUMBER = {1-2},
     PAGES = {1--82},
      ISSN = {0167-8019,1572-9036},
   MRCLASS = {05B05 (94B65)},
  MRNUMBER = {1192833},
MRREVIEWER = {Michael\ Klemm},
       DOI = {10.1007/BF00053379},
       URL = {https://doi.org/10.1007/BF00053379},
}

@article {LyubichShatalova05,
    AUTHOR = {Lyubich, Yuriĭ I. and Shatalova, Oksana A.},
     TITLE = {Polynomial functions on the classical projective spaces},
   JOURNAL = {Studia Math.},
  FJOURNAL = {Studia Mathematica},
    VOLUME = {170},
      YEAR = {2005},
    NUMBER = {1},
     PAGES = {77--87},
      ISSN = {0039-3223,1730-6337},
   MRCLASS = {33C55 (51M30)},
  MRNUMBER = {2142184},
MRREVIEWER = {Walter\ Schempp},
       DOI = {10.4064/sm170-1-4},
       URL = {https://doi.org/10.4064/sm170-1-4},
}

@article {Lyubich09,
    AUTHOR = {Lyubich, Yuriĭ I.},
     TITLE = {On tight projective designs},
   JOURNAL = {Des. Codes Cryptogr.},
  FJOURNAL = {Designs, Codes and Cryptography. An International Journal},
    VOLUME = {51},
      YEAR = {2009},
    NUMBER = {1},
     PAGES = {21--31},
      ISSN = {0925-1022,1573-7586},
   MRCLASS = {05B30},
  MRNUMBER = {2480685},
       DOI = {10.1007/s10623-008-9240-4},
       URL = {https://doi.org/10.1007/s10623-008-9240-4},
}

@techreport{Neumaier81,
    author      = {Neumaier, Arnold},
    title       = {Combinatorial configurations in terms of distances},
    institution = {Eindhoven University of Technology},
    year        = {1981},
    type        = {Memorandum},
    number      = {81-09 (Wiskunde)}
}

@misc{Okuda15,
  author       = {Okuda, Takayuki},
  title        = {Relation between spherical designs through a {H}opf map},
  howpublished = {Preprint, arXiv:1506.08414},
  year         = {2015}
}

@book {Schafer66,
    AUTHOR = {Schafer, Richard D.},
     TITLE = {An introduction to nonassociative algebras},
      NOTE = {Corrected reprint of the 1966 original},
 PUBLISHER = {Dover Publications, Inc., New York},
      YEAR = {1995},
     PAGES = {x+166},
      ISBN = {0-486-68813-5},
   MRCLASS = {17-01},
  MRNUMBER = {1375235},
}

@article {SeymourZalavsky84,
    AUTHOR = {Seymour, Paul D. and Zaslavsky, Thomas},
     TITLE = {Averaging sets: a generalization of mean values and spherical
              designs},
   JOURNAL = {Adv. in Math.},
  FJOURNAL = {Advances in Mathematics},
    VOLUME = {52},
      YEAR = {1984},
    NUMBER = {3},
     PAGES = {213--240},
      ISSN = {0001-8708},
   MRCLASS = {05B30 (26B15)},
  MRNUMBER = {744857},
MRREVIEWER = {J.\ J.\ Seidel},
       DOI = {10.1016/0001-8708(84)90022-7},
       URL = {https://doi.org/10.1016/0001-8708(84)90022-7},
}

@INPROCEEDINGS{Sloane...03,
    author={Hardin, Ronald H. and Sloane, Neil J. A. and Cara, Philippe},
    booktitle={Proceedings 2003 IEEE Information Theory Workshop}, 
    title={Spherical designs in four dimensions}, 
    year={2003},
    publisher={IEEE},
    pages={253-258},
    doi={10.1109/ITW.2003.1216742}
}

@incollection {Womersley17,
    AUTHOR = {Womersley, Robert S.},
     TITLE = {Efficient spherical designs with good geometric properties},
 BOOKTITLE = {Contemporary computational mathematics---a celebration of the
              80th birthday of {I}an {S}loan},
     PAGES = {1243--1285},
 PUBLISHER = {Springer, Cham},
      YEAR = {2018},
      ISBN = {978-3-319-72455-3; 978-3-319-72456-0},
   MRCLASS = {65D30 (65C05)},
  MRNUMBER = {3822282},
}

@incollection {Koning98,
    AUTHOR = {K\"{o}nig, Hermann},
     TITLE = {Cubature formulas on spheres},
 BOOKTITLE = {Advances in multivariate approximation ({W}itten-{B}ommerholz,
              1998)},
    SERIES = {Math. Res.},
    VOLUME = {107},
     PAGES = {201--211},
 PUBLISHER = {Wiley-VCH, Berlin},
      YEAR = {1999},
      ISBN = {3-527-40236-5},
   MRCLASS = {65D32 (41A55)},
  MRNUMBER = {1797231},
}

@article {Kuperberg06,
    AUTHOR = {Kuperberg, Greg},
     TITLE = {Numerical cubature from {A}rchimedes' hat-box theorem},
   JOURNAL = {SIAM J. Numer. Anal.},
  FJOURNAL = {SIAM Journal on Numerical Analysis},
    VOLUME = {44},
      YEAR = {2006},
    NUMBER = {3},
     PAGES = {908--935},
      ISSN = {0036-1429,1095-7170},
   MRCLASS = {65D32 (41A55)},
  MRNUMBER = {2231849},
       DOI = {10.1137/040615584},
       URL = {https://doi.org/10.1137/040615584},
}

\end{document}